\pdfoutput=1
\documentclass{amsart}
\usepackage[utf8]{inputenc}
\usepackage{amsmath,amssymb,amsthm}
\usepackage{bbm}
\usepackage{mathrsfs}
\usepackage{hyperref}
\usepackage{graphicx}
\usepackage{subfig}

\newcommand\longmapsfrom{\mathrel{\reflectbox{\ensuremath{\longmapsto}}}}

\def\articleTitle{Point electrode problems in piecewise smooth plane domains}
\def\shortArticleTitle{Point electrodes in piecewise smooth domains}

\hypersetup{
    pdftitle={\articleTitle},
    pdfauthor={Otto Seiskari},
    colorlinks=true,
    linkcolor=black,
    citecolor=black,
    filecolor=black,
    urlcolor=black,
}

\newtheorem{theorem}{Theorem}[section]
\newtheorem{corollary}[theorem]{Corollary}
\newtheorem{lemma}[theorem]{Lemma}
\newtheorem{definition}{Definition}[section]
\numberwithin{equation}{section} %

\theoremstyle{remark}
\newtheorem{remark}[theorem]{Remark}%
\newtheorem{example}[theorem]{Example}

\DeclareMathOperator{\supp}{{\rm supp}}	%
\def\R{\mathbb R}		%
\def\C{\mathbb C}		%
\def\d{\,\mathrm d}		%
\def\cont{\mathscr C}		%
\def\bg{{\mathbbm 1}}		%
\def\csubset{\subset\subset}	%
\def\wto{\xrightarrow{w*}}	%
\newcommand{\mb}[1]{{\boldsymbol{#1}}}

\begin{document}

\title[\shortArticleTitle]{\articleTitle}

\author{Otto Seiskari}
\address{Aalto University, Department of Mathematics and Systems Analysis, FI-00076 Aalto, Finland}
\email{otto.seiskari@aalto.fi}


\begin{abstract}
Conductivity equation is studied in piecewise smooth plane domains and with measure-valued current patterns (Neumann boundary values). 
This allows one to extend the recently introduced concept of \emph{bisweep data} to piecewise smooth domains, which yields a new partial data result for Calderón inverse conductivity problem.
It is also shown that bisweep data are (up to a constant scaling factor) the Schwartz kernel of the relative Neumann-to-Dirichlet map.
A numerical method for reconstructing the supports of inclusions from discrete bisweep data is also presented.
\end{abstract}

\maketitle

\section{Introduction}

In \emph{electrical impedance tomography} (EIT) \cite{uhlmannreview}\cite{borceareview}, the conductivity equation
\begin{equation}
\label{eq:conductivity:equation}
	\nabla \cdot (\sigma \nabla u) = 0 \ \mbox{ in }D, \qquad
	\frac{\partial u}{\partial \nu} = f \ \mbox{ on }\partial D,
\end{equation} 
where $D$ is a bounded domain in $\R^2$ or $\R^3$, is studied.
A fundamental question, the so-called Calderón inverse conductivity problem, is whether a measurable conductivity $\sigma$ can be determined from the Neumann-to-Dirichlet map (current-to-voltage map) $\Lambda_\sigma : f \mapsto u|_{\partial D}$, or equivalently, its inverse, the Dirichlet-to-Neumann map $\Lambda_\sigma^{-1}$. 
In the plane, $D \subset \R^2$, a positive answer was given by Astala and Päivärinta \cite{astalapaivarinta2006}
under the assumption that $D$ has a connected complement and $\sigma$ is essentially bounded away from zero and infinity (see equation \ref{eq:sigma:bounds} below). %
This followed a number of uniqueness proofs in case of more regular conductivities, e.g., \cite{Nachman96}\cite{brown1997uniqueness}. The problem has also been extensively studied in dimensions $n \geq 3$, see \cite{uhlmannreview} and the references therein.

The conductivity equation is typically studied with boundary currents $f \in H^{-1/2}(\partial D)$, but from a certain point of view it is also natural to consider more singular distributional values, in particular, various point distributions on the boundary of the object.
The approach is inspired by the concept of \emph{point electrodes} \cite{pem},
which can be used as an alternative approximate physics model in place of the so-called \emph{continuum forward model}, the basis of many theoretical results and numerical methods in EIT. Another commonly used framework is the \emph{complete electrode model} (CEM),
which is highly realistic but its theoretical properties remain less studied \cite{eit99}\cite{borceareview}. %

We study measurements of the form
\begin{equation}
\label{eq:def:w:angle:brackets}
	w(x,y,p,q) = \langle \delta_x - \delta_y, (\Lambda_\sigma - \Lambda_\bg)(\delta_p - \delta_q) \rangle,
\end{equation}
where $\Lambda_\bg$ is the Neumann-to-Dirichlet map for Laplace equation in $D$. In domains with a smooth boundary $\partial D$, this type of distributional measurements can be tangled using the techniques from \cite{lionsmagenes} (cf. \cite{convexbss}).
This paper presents an alternative approach, based on \cite{kral1980}, that is also applicable in piecewise smooth domains.
It appears that the point measurements $w$ are well-defined provided that
\begin{subequations}
\label{eq:feasible:sigma}
\begin{align}
\label{eq:sigma:bounds}
	&C \geq \sigma \geq c > 0 \mbox{ a.e. in }D, \\ 
\label{eq:sigma:boundary:homogeneity}
	&\supp(1-\sigma) \csubset D.
\end{align}
\end{subequations}
Thus, in addition to being bounded, $\sigma$ is assumed to be homogeneous near the boundary $\partial D$.

In dimension $n=2$, there are recent results considering Calderón problem with partial data; is it possible to recover $\sigma$ from partial knowledge of $\Lambda_\sigma$ (or $\Lambda_\sigma^{-1}$)? It was shown by Imanuvilov, Uhlmann, and Yamamoto \cite{imanuvilov2010} that if $\sigma$ is, \emph{a priori}, known to have smoothness of class $\cont^{4,\alpha}, \alpha > 0$ (this assumption is relaxed in \cite{imanuvilov2012}), then it can be recovered from the knowledge of $\{ (v, (\Lambda_\sigma^{-1} v)|_E) \,:\, \supp v \subset E \}$ for any (relatively) open subset $E$ of a piecewise smooth boundary $\partial D$. A similar result for the Neumann-to-Dirichlet map $\Lambda_\sigma$ and conductivities $\sigma \in W^{3,p}(D)$, $p > 2$ in smooth domains $D$ is formulated in \cite{imanuvilovpreprint}.

In this paper, we show that if the global smoothness assumptions in \cite{imanuvilov2010}\cite{imanuvilov2012}\cite{imanuvilovpreprint} are exchanged for another, arguably unwanted, restriction \eqref{eq:sigma:boundary:homogeneity}, a measurable conductivity $\sigma$ is uniquely determined by various types of partial data.
In particular, a result in \cite{bisweep} is generalized by showing that $\sigma$ is recovered from the knowledge of $w(x,y,p,q)$, $x,y \in \Gamma$, $p,q \in \Xi$ for arbitrary countably infinite $\Gamma, \Xi \subset \partial D$. In addition, $\sigma$ can be recovered from measurements of the type $(\Lambda_\sigma f)|_{W}$, $\supp f \subset V$, where $V, W \neq \emptyset$ are arbitrary, possibly disjoint, relatively open subsets of $\partial D$.

Many numerical reconstruction methods in EIT are easiest to apply if $D = B$ is the unit disk. Together with numerical conformal mappings, the concept of \emph{bisweep data} \cite{bisweep}, $\varsigma_\sigma(x,y) := w(x,y,x,y)$, offers means of using any such reconstruction algorithm in other piecewise smooth, simply connected plane domains.
In Section~\ref{sec:rec:numerics}, we formulate a method for applying unit-disk-based reconstruction algorithms in other domains, and demonstrate it with the factorization method \cite{bruhl2001factorization}.

The main results of this article are outlined in the next section and proven in Section~\ref{sec:proofs}.
For completeness, auxiliary theorems about conformal maps are given in an appendix.

\section{Setting and main results}

\begin{definition}[Piecewise smooth domain]
Let $D \subset \R^2$ be a bounded, simply connected Lipschitz domain %
with a simple boundary $\partial D$ consisting of a finite number $m$ of $\cont^{1,\alpha}$ smooth %
arcs (for some $0 < \alpha \leq 1$). We call such $D$ a piecewise $\cont^{1,\alpha}$ smooth plane domain.
\end{definition}

Throughout this paper, the symbol $D$ is used to denote an arbitrary piecewise $\cont^{1,\alpha}$ smooth domain, which is also identified with the corresponding subset of $\C$ when necessary.
Notice that, as a Lipschitz domain, $D$ cannot have any cusps, and the surface (Lebesgue) measure $s$ is well-defined. The symbol $\nu$ is used for the outward unit normal vector of $\partial D$, which is defined $s$-almost everywhere (cf., e.g., \cite[§1.5]{grisvard}). %
Sobolev and Lebesgue spaces in $D$ (resp., on $\partial D$) are defined with respect to the Lebesgue measure in $\R^2$ (resp., $s$).

Let $\sigma \in L^\infty(D)$ satisfy \eqref{eq:feasible:sigma}, which is also assumed in what follows. For any $f \in H^{-1/2}_\diamond(\partial D)$, let $u^f \in H^1(D)/\R$ denote the unique\footnote{%
	Unique up to an addition of a constant, that is, in $H^1(D)/\R := \{ \{ u + C \,:\, C \in \R \} \,:\,u \in H^1(D) \}$ (cf. \cite[§1.1.7]{necas}\cite[§1.1.13]{mazja1985}).
	The solution exists if and only if the Neumann boundary value has zero mean, that is, $f \in H^{-1/2}_\diamond(\partial D) := \{ g \in H^{-1/2}(\partial D) \,:\, \langle g, 1 \rangle = 0 \}$, which is the dual space of $H^{1/2}(\partial D)/\R$. The brackets denote dual evaluation.
	For the definition of $H^{1/2}(\partial D)$, see, e.g., \cite{grisvard}\cite{necas}, where this space is denoted $W^{1/2}_2(\partial D)$.
} 
solution to the weak form of \eqref{eq:conductivity:equation}  \cite{astalapaivarinta2006}\cite{borceareview}:
\begin{equation}
\label{eq:def:u:H1}
	\int_D \sigma \nabla u^f \cdot \nabla v \d x = \langle f, \gamma v \rangle \qquad \mbox{ for all }v \in H^1(D)/\R,
\end{equation} 
where $\gamma : H^1(D)/\R \to H^{1/2}(\partial D)/\R$ is the continuous Dirichlet trace operator (see \cite[§2.5.4]{necas}).
\emph{Neumann-to-Dirichlet map} is the bounded and self-dual %
operator $\Lambda_\sigma : H^{-1/2}_\diamond(\partial D) \rightarrow H^{1/2}(\partial D)/\R$ defined by $f \mapsto \gamma u^f$. 
We also study the background Neumann-to-Dirichlet map, $\Lambda_\bg : H^{-1/2}_\diamond(\partial D) \rightarrow H^{1/2}(\partial D)/\R$, $f \mapsto \gamma u^f_\bg$, where $u^f_\bg \in H^1(D)/\R$ is the unique solution to \eqref{eq:def:u:H1} with $\sigma \equiv 1$, i.e., Laplace equation with Neumann boundary value $f$.

By $\cont'(\partial D)$, we denote the Banach space of all finite, real, signed Borel measures supported in $\partial D$, equipped with the total variation norm (cf. \cite{kral1980}). 
It is the dual space of $\cont(\partial D)$ \cite[Chapter IV, Theorem 6.2]{dunfordschwartz}. %
The subspace of measures $\mu \in \cont'(\partial D)$ such that $\mu(\partial D) = 0$ is denoted $\cont_\diamond'(\partial D)$ and it can be identified with the dual of $\cont(\partial D)/\R$, the space of continuous functions on $\partial D$ defined up to a constant, which is complete %
and can be supplied with the norm (cf. \cite[§1.1.7]{necas})
\[
	\| f \|_{\cont(\partial D)/\R}
	= \inf_{g \in f} \| g \|_{\cont(\partial D)}
	= \inf_{C \in \R}\sup_{x \in \partial D} |f(x) + C|.
\]

The functions %
in $\cont_\diamond(\partial D) = \{ f \in \cont(\partial D) : \int_{\partial D} f \d s = 0 \}$ are dense in the \emph{weak* topology} 
of $\cont'_\diamond(\partial D)$, that is, for any $\mu \in \cont'_\diamond(\partial D)$, there exists a sequence $\{ \mu_j \} \subset \cont_\diamond(\partial D)$ 
such that (cf. Lemma~\ref{lemma:approximation})
\[
	\int_{\partial D} f \mu_j \d s \rightarrow \int_{\partial D} f \d \mu \qquad \mbox{ for all }f \in \cont(\partial D)/\R,
\]
which is denoted $\mu_j \wto \mu$. 

Let $\mu \in \cont'_\diamond(\partial D)$. As a generalization\footnote{%
	If $\mu \in H^{-1/2}_\diamond(\partial D) \cap \cont_\diamond'(\partial D)$,
	that is,  $\mu \in \cont'_\diamond(\partial D)$ and $|\int_{\partial D} \varphi \d \mu| \leq C \|\varphi\|_{H^{1/2}(\partial D)/\R}$, in which case $\mu$ can be identified with the continuous extension of $\varphi \mapsto \int_{\partial D} \varphi \d \mu$ to $H^{-1/2}_\diamond(\partial D)$, the definitions \eqref{eq:def:u:H1} and \eqref{eq:def:umu} coincide.}
of the conductivity equation for measure-valued boundary currents, the weak Neumann problem %
of finding a function $u^\mu \in W^{1,1}(D)/\R = \{ \{ u + C \,:\, C \in \R \}\,:\, u, |\nabla u| \in L^1(D)  \}$ such that
\begin{equation}
\label{eq:def:umu}
	\int_D \sigma \nabla u^\mu \cdot \nabla \varphi \d x = \int_{\partial D} \varphi \d \mu \qquad \mbox{for all }\varphi \in \cont^\infty(\overline D)
\end{equation} 
is studied.
The corresponding background problem is
\begin{equation}
\label{eq:def:u1mu}
	\int_D \nabla u_\bg^\mu \cdot \nabla \varphi \d x = \int_{\partial D} \varphi \d \mu \qquad \mbox{for all }\varphi \in \cont^\infty(\overline D),
\end{equation}
and, as shown in Lemma~\ref{lemma:u1:Neumann:function}, one representative of the background solution is given by
\begin{equation}
\label{eq:u1:Neumann:function}
	u^\mu_\bg(x) = \int_{\partial D} N(\Phi(x),\Phi(y)) \d \mu(y),
\end{equation}
where $N$ is the Neumann--Green function of the unit disk $B$: %
\begin{equation}
\label{eq:unit:disk:Neumann:function} 
	N(x,y) = %
	\begin{cases}
		-\frac1{2\pi} \left( \log|y-x| + \log\left| \frac{y}{|y|} - |y|x \right| \right) &\mbox{ if }y \neq 0, \\
		-\frac1{2\pi} \log|x|	&\mbox{ if }y = 0
	\end{cases}
\end{equation}
and $\Phi : D \to B$ is a conformal map. 

The main results of this paper are connected to the following concept.

\begin{theorem}\label{thm:NtoD:continuity}
There exists a unique, linear, and bounded operator, $\Lambda_\sigma - \Lambda_\bg : \cont'_\diamond(\partial D) \to \cont(\partial D)/\R$, called the \emph{relative Neumann-to-Dirichlet map}, such that for all $\mu, \eta \in \cont'_\diamond(\partial D)$,
\begin{equation}
\label{eq:def:relative:NtoD:and:Q}
	\int_{\partial D} [(\Lambda_\sigma - \Lambda_\bg) \mu] \d \eta = Q_\sigma(\mu,\eta) := \int_{D} (1-\sigma) \nabla u^\mu \cdot \nabla u^\eta_\bg \d x.
\end{equation}
The map is also self-dual %
in the sense that the bilinear form $Q_\sigma : \cont'_\diamond(\partial D) \times \cont'_\diamond(\partial D) \to \R$ is symmetric: $Q_\sigma(\mu,\eta) = Q_\sigma(\eta,\mu)$.
Furthermore,
\begin{equation}
\label{eq:NtoD:weak:convergence}
	Q_\sigma(\mu,\eta)
	= \lim_{j \to \infty} \langle \eta_j, (\Lambda_\sigma - \Lambda_\bg) \mu_j \rangle
	,
\end{equation} 
where $\{ \mu_j \}, \{ \eta_j \} \subset \cont_\diamond(\partial D)$ are any sequences such that $\mu_j \wto \mu$ and $\eta_j \wto \eta$.
\end{theorem}

The function $(\Lambda_\sigma - \Lambda_\bg)\mu$ is given by the Dirichlet trace $\gamma (u^\mu - u^\mu_\bg)$, where $u^\mu - u^\mu_\bg \in H^1(D)/\R$ (Lem\-ma~\ref{lemma:relative:NtoD:is:trace:of:w}). Thus the restriction of $\Lambda_\sigma - \Lambda_\bg$ to $H^{-1/2}_\diamond(\partial D) \cap \cont'_\diamond(\partial D)$ coincides
with the difference of $\Lambda_\sigma$ and $\Lambda_\bg$ as defined on page \pageref{eq:def:u:H1}, which justifies the notation.
In smooth domains $D$, the definition \eqref{eq:def:relative:NtoD:and:Q} likewise coincides with the continuous extension $\Lambda_\sigma - \Lambda_\bg : H^{-s}_\diamond(\partial D) \to H^s(\partial D)/\R$ of the relative map  to distributional Sobolev spaces with $s > 1/2$ (Lemma~\ref{lemma:same:NtoD:definitions}). 
The continuity of $\Lambda_\sigma - \Lambda_\bg : \cont'_\diamond(\partial D) \to \cont(\partial D)/\R$ is proven with the aid of a factorization in Theorem~\ref{thm:NtoD:factorization}.

\begin{definition}\label{def:bisweep:data}
The function $\varsigma_\sigma : \partial D \times \partial D \rightarrow \R$,
\[
	\varsigma_\sigma(x,y)
	= Q_\sigma(\delta_x - \delta_y, \delta_x - \delta_y)
	,
\]
where $\delta_x : f \mapsto f(x)$, is called the \emph{bisweep data} of $\sigma$.
\end{definition}

Bisweep data can be seen as a \emph{point electrode model} of the following two-electrode EIT measurement: the voltage required to maintain unit current between electrodes at $x$ and $y$ is measured. The same measurement is then performed with a homogeneous reference object of shape $D$ (i.e., $\sigma = 1$ in $D$). The difference between these two measurements, as a function of the electrode positions, is modeled by the bisweep data. See \cite{pem} and \cite{sweep} for a rigorous treatment of this argument in smooth domains.

Theorem \ref{thm:NtoD:continuity} allows one to extend the partial data result \cite{bisweep} for Calderón problem to piecewise smooth plane domains:
\begin{theorem}\label{thm:partial:calderon}
Let $\Gamma \subset \partial D$ be countably infinite and $\sigma$ satisfy \eqref{eq:feasible:sigma}. The knowledge of $\varsigma_\sigma$ on $\Gamma \times \Gamma$ uniquely determines $\sigma$.
\end{theorem}

This can be generalized as follows

\begin{theorem}\label{thm:partial:calderon:w}
Let $\Gamma, \Xi$ be (possibly disjoint) countably infinite subsets of $\partial D$.
The knowledge of
\begin{equation}
\label{eq:def:four:electrode:function}
	w(x,y,p,q) := Q_\sigma(\delta_x - \delta_y, \delta_p - \delta_q)
\end{equation}
for all $x,y \in \Gamma$, $p,q \in \Xi$ uniquely determines $\sigma$.
\end{theorem}

The above also relates to the next corollary, which can alternatively be proven directly using a unique continuation argument. (Cauchy data for Laplace's equation in some neighborhood of the boundary is always available on a certain part of $\partial D$.)

\begin{corollary}\label{corollary:L2:partial:calderon}
Let $V, W$ be (possibly disjoint) non-empty relatively open subsets of $\partial D$ and denote by $\cont_\diamond(V) \subset \cont_\diamond(\partial D)$ (resp., $\cont_\diamond(W)$) %
the continuous mean-free functions supported in $V$ (resp., $W$). The knowledge of
\[
	\left\{ \left(f,g, \langle f, \Lambda_\sigma g \rangle\right) \,:\, f \in \cont_\diamond(V),\;g \in \cont_\diamond(W) \right\},
\]
uniquely determines any conductivity $\sigma$ that satisfies \eqref{eq:feasible:sigma}. In other words, if
$
	\langle f, \Lambda_\sigma g \rangle$ $=$ $\langle f, \Lambda_{\tilde \sigma} g \rangle
$
for all $f \in \cont_\diamond(V)$, $g \in \cont_\diamond(W)$, then $\sigma = \tilde \sigma$ almost everywhere in $D$.
\end{corollary}

These results are based on the fact that, in the unit disk $D = B$, bisweep data and its generalization in Theorem~\ref{thm:partial:calderon:w} are jointly analytic functions. Bisweep data also comprise (up to a scaling factor) the \emph{Schwartz kernel} %
of the relative Neumann-to-Dirichlet map in the sense that %

\begin{theorem}\label{thm:bisweep:NtoD:kernel}
For any $\mu, \eta \in \cont'_\diamond(\partial D)$,
\begin{equation*}
	\int_{\partial D} [(\Lambda_\sigma - \Lambda_\bg) \mu] \d \eta = -\frac12 \iint\limits_{\partial D \times \partial D} \varsigma_\sigma \d \mu \times \eta,
\end{equation*}
and $\varsigma_\sigma \in \cont^{\alpha}(\partial D \times \partial D)$ for some $\alpha > 0$.
\end{theorem}

Many of these relations can be proven by generalizing the corresponding result from the unit disk $B$ (or some other smooth domain) to an arbitrary piecewise smooth plane domain $D$ using the fact that \emph{point current sources} $\mu = \sum_{j=1}^m \alpha_j \delta_{x_j}$ are moved naturally by conformal mappings (cf. \cite{sweep}). 
In particular, bisweep data provide a natural method of ``transporting'' relative Neumann-to-Dirichlet maps between different domains. Namely, if $\tilde \sigma = \sigma \circ \Phi^{-1}$, where $\Phi$ is a conformal map from $D$ to the unit disk $B$, then
\begin{equation}
\label{eq:bisweep:conformal}
	\varsigma_\sigma(x,y) = \varsigma_{\tilde \sigma}(\Phi(x),\Phi(y))
\end{equation}
for all $x,y \in \partial D$. %
Notice that, since $\partial D$ is a Jordan curve, $\Phi$ extends to a homeomorphism between $\overline D$ and $\overline B$. In fact, $\Phi$ is Hölder continuous and its further smoothness properties are determined by the smoothness of the arcs of $\partial D$ and its vertex angles (cf. \cite[Ch. 3]{pommerenke1992} and Theorem~\ref{thm:conformal:sharp:angle:intermediary}).

\section{Proofs of the results}\label{sec:proofs}

It follows from Poincaré inequality %
and the boundedness assumptions on $\sigma$ that the left hand side of \eqref{eq:def:u:H1} defines a bounded and coercive bilinear form in the Hilbert space $H^1(D)/\R$. The right hand side is a continuous functional in $(H^1(D)/\R)'$ due to the trace theorem \cite[§2.5.4]{necas}.
Hence the unique solvability of \eqref{eq:def:u:H1} and continuity with respect to the data follows readily from Riesz representation theorem. (cf. \cite[Chapter VII, §1.2.2]{dautraylions})

In case of \eqref{eq:def:umu} and \eqref{eq:def:u1mu} the above technique fails because the distributions defined by the right hand sides are not generally bounded in $(H^1(D))'$ (since $\cont'(\partial D) \nsubseteq H^{-1/2}(\partial D)$). %
To the best knowledge of the author, the distributional theory studied by, e.g., Lions, Magenes, and Ne{\v c}as \cite{lionsmagenes}\cite{necas} is not directly applicable either, due to the higher regularity requirements for the domain $D$.

This section shows an alternative approach, based on the work of Král \cite{kral1980}, for solving these problems.

\subsection{Weak Neumann problems}\label{sec:weak:neumann}

The fact that the weak background Neumann problem \eqref{eq:def:u1mu} has a unique (harmonic) solution $u^\mu_\bg \in \cont^\infty(D)/\R$ such that $|\nabla u^\mu_\bg| \in L^1(D)$ follows from \cite{kral1980} provided that certain geometric assumptions on $D$ are satisfied. It is stated in, e.g., \cite{medkova1997third} that piecewise $\cont^{1,\alpha}$ smoothness is sufficient. Regarding the conductivity equation, there holds

\begin{lemma}\label{lemma:umu:solution:and:w}
The weak conductivity equation \eqref{eq:def:umu} has a unique solution $u^\mu \in W^{1,1}(D)/\R$, given by $u^\mu = u^\mu_\bg + w^\mu$, where $w^\mu \in H^1(D)/\R$ satisfies
\begin{equation}
\label{eq:def:w}
	\int_D \sigma \nabla w^\mu \cdot \nabla \varphi \d x
	= \int_D (1-\sigma) \nabla u_\bg^\mu \cdot \nabla \varphi \d x
	\qquad \mbox{for all }\varphi \in H^1(D)/\R.
\end{equation}
\end{lemma}

\begin{proof}
Due to Poincaré inequality \cite[§1.1.7]{necas} and the boundedness assumptions in \eqref{eq:feasible:sigma}, the left hand side of \eqref{eq:def:w} is bounded and coercive.
The right hand side defines a continuous functional $\varphi \mapsto \int_D (1-\sigma) \nabla u_\bg^\mu \cdot \nabla \varphi \d x$ in $(H^1(D)/\R)'$,
for $u_\bg^\mu$ is smooth in $\supp(1-\sigma)$.
As a result, there exists a unique solution $w^\mu$. Furthermore, $\cont^\infty(\overline D)/\R$ is dense in $H^1(D)/\R$ (cf. \cite[§1.1.6]{mazja1985}) %
so \eqref{eq:def:w} is satisfied for all $\varphi \in H^1(D)/\R$ if and only if it holds for all $\varphi \in \cont^\infty(\overline D)/\R$.
It follows that $u_\bg^\mu + w^\mu$ is the unique solution to \eqref{eq:def:umu}.

Notice that $u^\mu_\bg$ is also in $L^1(D)/\R$ by \cite[§1.1.11]{mazja1985} and therefore $u^\mu_\bg$, $u^\mu \in W^{1,1}(D)/\R$.
% TODO: remark on uniqueness of $u^\mu$
\end{proof}

The following two lemmas state the relationship between the weak formulations presented above and a distributional Sobolev space formulation based on trace theorems, utilized in, e.g., \cite{bisweep}.

\begin{lemma}\label{lemma:same:u1}
Let $D$ be a $\cont^\infty$ smooth domain, $\mu \in \cont'_\diamond(\partial D)$ arbitrary, and $s < 1$. 
Then $u_\bg = u_\bg^\mu$ solves \eqref{eq:def:u1mu} if and only if $u_\bg \in X^{s}$ %
and
\begin{equation}
\label{eq:def:u1:distributional}
	\Delta u_\bg = 0 \ \mbox{ in }D, \qquad \gamma_1 u_\bg = \mu \ \mbox{ on }\partial D,
\end{equation}
where $\gamma_1 : X^{s} \to H^{s-3/2}(\partial D)$ is the continuous extension of the normal derivative in, e.g., the Banach space $X^{s} = \{ u \in H^{s}(D)/\R \,:\, \Delta u \in L^2(D) \}$ equipped with the graph norm (cf. \cite[Chapter 2, Remark 7.2]{lionsmagenes}).
\end{lemma}

\begin{proof}
Assume that $u_\bg = u_\bg^\mu$ solves \eqref{eq:def:u1mu} and $s \leq -2$.
Clearly, $\Delta u_\bg = 0$ in $D$. According to \cite[Theorem 1.1.6.2]{mazja1985} (and its proof), there exists a sequence $\{ u_j \} \subset \cont^\infty(\overline D)$ that converges to (an arbitrary representative of) $u_\bg$ w.r.t. %
the norm 
\[
	\| v \|_* := \| v \|_{L^1(D)} + \||\nabla v|\|_{L^1(D)} + \|\Delta v\|_{L^2(D)}.
\]
Due to Sobolev's embedding theorem, %
\[
	\| v \|_{H^{-2}(D)} = \sup_{\|\varphi\|_{H^2_0(D)}=1} \left|\int_D v \varphi \d x \right|
	\leq \sup_{\|\varphi\|_{H^2_0(D)}=1} \| v \|_{L^1(D)} \| \varphi \|_{\cont(\overline D)}
	\leq C \| v \|_{L^1(D)}
\]
and consequently, $u_j \rightarrow u_\bg \in X^{-2}$, $\frac{\partial u_j}{\partial \nu} \rightarrow \gamma_1 u_\bg$ in $H^{-7/2}(\partial D)$. 
For any $\varphi \in \cont^\infty(\partial D)$, let $\tilde \varphi$ be an arbitrary extension of $\varphi$ to $\cont^\infty(\overline D)$. Then, by Green's first identity,
\[
\begin{split} 
	\langle \gamma_1 u_\bg, \varphi \rangle
	&= \lim_{j\to \infty} \langle \tfrac{\partial u_j}{\partial \nu}, \varphi \rangle
	= \lim_{j\to \infty} \int_{\partial D} \tfrac{\partial u_j}{\partial \nu} \varphi \d s
	= \lim_{j\to \infty} \int_D (\nabla u_j \cdot \nabla \tilde \varphi + \tilde \varphi \Delta u_j) \d x \\
	&= \int_D \nabla u_\bg \cdot \nabla \tilde \varphi \d x 
	= \int_{\partial D} \varphi \d \mu,
\end{split}
\]%
which means that $\gamma_1 u_\bg = \mu$. Thus $u_\bg$ solves \eqref{eq:def:u1:distributional}.

Since the solution $\tilde u_\bg$ to \eqref{eq:def:u1:distributional} with $s \leq -2$ is unique, it follows that it can always be identified with a function (equivalence class) $u_\bg$ that satisfies \eqref{eq:def:u1mu}. If $u_\bg$ solves \eqref{eq:def:u1:distributional} for $-2 < s < 1$, this also remains true for $s \leq -2$. Furthermore, a unique solution to \eqref{eq:def:u1:distributional} exists for any $s < 1$ (because $\cont'(\partial D) \subset H^{-1/2-\epsilon}(\partial D)$ for any $\epsilon > 0$), which proves the general claim.
\end{proof}

In smooth domains $D$, it is possible to extend the operator $\Lambda_\sigma - \Lambda_\bg : H^{-1/2}_\diamond(\partial D) \rightarrow H^{1/2}(\partial D)/\R$ (as defined on page \pageref{eq:def:u:H1}) to a continuous mapping between the Sobolev spaces $H^{-s}_\diamond(\partial D)$ and $H^s(\partial D)/\R$ for any $s \in \R$ (cf. \cite{convexbss}). The next lemma shows that the result coincides with the definition \eqref{eq:def:relative:NtoD:and:Q}.

\begin{lemma}\label{lemma:same:NtoD:definitions}
Let $D$ be a $\cont^\infty$ smooth domain and $s > 1/2$. Then
\begin{equation}
\label{eq:NtoD:integral:formula:sobolev}
	\langle \eta, (\Lambda_\sigma - \Lambda_\bg) \mu \rangle = Q_\sigma(\mu,\eta) = Q_\sigma(\eta,\mu)
\end{equation}
where $\mu, \eta \in \cont'_\diamond(\partial D)$ are arbitrary and $\Lambda_\sigma - \Lambda_\bg : H^{-s}_\diamond(\partial D) \rightarrow H^s(\partial D)/\R$.
\end{lemma}

\begin{proof}
Fix an arbitrary $s > 1/2$ and let $U \csubset D$ be such that $\supp(1-\sigma) \csubset U$. 
For any $\phi \in H^{-s}_\diamond(\partial D)$, let $u^\phi_\bg$ be the unique solution to \eqref{eq:def:u1:distributional} with boundary value $\phi$, which is clearly equivalent to definitions \eqref{eq:def:u:H1} and \eqref{eq:def:u1mu} for smooth $\phi$.
Moreover, for any smooth functions $\phi, \varphi \in \cont^\infty_\diamond(\partial D)$, the ``Green's formulas''
\begin{equation}
\label{eq:NtoD:integral:formula:smooth}
\begin{split}
	\langle \varphi, (\Lambda_\sigma - \Lambda_\bg) \phi \rangle
	&= \int_D (1-\sigma) \nabla u_\bg^\phi \cdot \nabla u^\varphi \d x
	= \int_D (1-\sigma) \nabla u^\phi \cdot \nabla u_\bg^\varphi \d x \\
	&= \int_U (1-\sigma) \nabla (u_\bg^\phi + w^\phi) \cdot \nabla u_\bg^\varphi \d x
\end{split}
\end{equation}
where $w^\phi$ solves \eqref{eq:def:w}, follow readily from the definitions of $\Lambda_\sigma$, $\Lambda_\bg$ on page \pageref{eq:def:u:H1} and Lemma~\ref{lemma:umu:solution:and:w}.

Due to standard elliptic regularity theory (cf., e.g., \cite[Appendix]{convexbss}), the operator $\phi \mapsto u_\bg^\phi|_U : H^{-s}_\diamond(\partial D) \to H^1(U)/\R$ is continuous. The mapping $u_\bg^\phi|_U \mapsto w^\phi : H^1(U)/\R \to H^1(D)/\R$ is also bounded (see the proof of Theorem~\ref{thm:NtoD:factorization} for details).
Consequently, \eqref{eq:NtoD:integral:formula:smooth} is well-defined for any $\phi$, $\varphi \in H^{-s}_\diamond(\partial D)$ and it yields the unique continuous extension $\Lambda_\sigma - \Lambda_\bg : H^{-s}_\diamond(\partial D) \to H^s(\partial D)/\R$. It follows from Lemma~\ref{lemma:same:u1} that the extension satisfies \eqref{eq:NtoD:integral:formula:sobolev} for any $\mu$, $\eta \in \cont'_\diamond(\partial D) \subset H^{-s}_\diamond(\partial D)$.
\end{proof}

In the next section, the formulas \eqref{eq:def:umu}, \eqref{eq:def:u1mu} need to be considered with less smooth test functions $\varphi$. The following lemma states that they remain valid as long as $\varphi$ is Lipschitz.

\begin{lemma}\label{lemma:variation:c1uniform}
If $u = u^\mu$ solves \eqref{eq:def:umu} then \eqref{eq:def:umu} in fact holds for all 
$\varphi \in \cont^{0,1}(\overline D)$.
\end{lemma}

\begin{proof}
Since $\varphi$ is Lipschitz continuous, there exists an extension $\tilde \varphi : \R^d \rightarrow \R$ such that $\tilde \varphi|_{\overline D} = \varphi$ and $\||\nabla \tilde \varphi |\|_{L^\infty(\R^d)} \leq C$ (cf. \cite[§5.8.2]{evans1998} and \cite{mcshane1934extension}).
Let $\varphi_\epsilon = h_\epsilon \ast \tilde \varphi \in \cont^\infty(\R^d)$, where $h_\epsilon \in \cont^\infty_0(B_\epsilon(0))$ is a \emph{mollifier} %
and $\ast$ denotes convolution \cite[§C.4]{evans1998}.

Denote $D_\eta = \{ x \in D \;:\; {\rm dist}(x,\partial D) > \eta \}$ (where $\eta > \epsilon$). Then
\begin{align*}
	&\left|\int_D \sigma \nabla u \cdot \nabla \varphi_\epsilon \d x
	- \int_D \sigma \nabla u \cdot \nabla \varphi \d x\right| 
	=  \left|\int_D \sigma \nabla u \cdot (\nabla \tilde \varphi_\epsilon - \nabla \varphi) \d x\right| 
	\\
	&\leq \|\sigma\|_{L^\infty(D)} \left( 2 C \int_{D \setminus D_\eta}|\nabla u| \d x + \int_{D_\eta} |\nabla u| |\nabla \varphi - \nabla \varphi_\epsilon| \d x\right),
\end{align*}
which can be made arbitrarily small by first choosing a suitable $\eta > 0$ and then $\epsilon < \eta$,
because $\nabla \varphi_\epsilon \to \nabla \varphi \in L^2(D_\eta)$ and $\nabla u \in L^2(D_\eta)$.
On the other hand,
\[
	\int_D \sigma \nabla u \cdot \nabla \varphi_\epsilon \d x
	= \int_{\partial D} \varphi_\epsilon \d \mu
	\quad\xrightarrow{\epsilon \rightarrow 0}\quad
	\int_{\partial D} \varphi \d \mu
\]
due to the uniform continuity of $\varphi$.
\end{proof}

\subsection{Conformally mapped Neumann problems}

A key ingredient in the partial data results in Theorems \ref{thm:partial:calderon} and \ref{thm:partial:calderon:w} is the ability to transform to an equivalent problem in the unit disk. To this end, the next theorem describes how the solution to the weak conductivity equation transforms under conformal mappings between piecewise smooth domains. 

\begin{theorem}\label{thm:pull:back:var}
Let $D, E$ be piecewise $\cont^{1,\alpha}$ smooth plane domains and $\Phi : D \rightarrow E$ a conformal map. If $u^\mu$ solves \eqref{eq:def:umu}, then $\tilde u = u^\mu \circ \Phi^{-1}$ solves
\begin{equation}
\label{eq:pull:back:var}
	\int_{E} \tilde \sigma \nabla \tilde u \cdot \nabla \tilde \varphi \d x = \int_{\partial E} \tilde \varphi \d \tilde \mu \qquad \mbox{ for all }\tilde \varphi \in \cont^\infty(\overline E),
\end{equation}
where
\[
	\tilde \sigma = \sigma \circ \Phi^{-1}, \qquad
	\tilde \mu \in \cont'_\diamond(\partial E)  :  \tilde \varphi \mapsto \int_{\partial D} \tilde \varphi \circ \Phi \d \mu.
\]
\end{theorem}

\begin{proof}
Let us first study the case when $\Phi \in \cont^1(\overline D)$ (i.e., $\Phi$ can be extended to such a function). 
For any $\tilde \varphi \in \cont^\infty(\overline E)$, let $\varphi = \tilde \varphi \circ \Phi$. Then $\varphi \in \cont^1(\overline D)$ and, by Lemma~\ref{lemma:variation:c1uniform},
\begin{align*}
	\int_{\partial E} \tilde \varphi \d \tilde \mu
	= \int_{\partial D} \varphi \d \mu
	= \int_{D} \sigma \nabla u^\mu \cdot \nabla \varphi \d x
	= \int_{E} \tilde \sigma \nabla \tilde u \cdot \nabla \tilde \varphi \d x,
\end{align*}
where the last step follows from Lemma~\ref{lemma:conformal:substitution}.

Conversely, if $\tilde u$ solves \eqref{eq:pull:back:var}, then it must hold that $\tilde u = u^\mu \circ \Phi^{-1}$, because the solution to \eqref{eq:pull:back:var} is known to be unique (up to addition of a constant function). Thus $\tilde u \circ \Phi$ solves \eqref{eq:def:umu}.

Due to Theorem~\ref{thm:conformal:sharp:angle:intermediary}, there exists, for some $\alpha' > 0$, piecewise $\cont^{1,\alpha'}$ smooth domains $D'$, $E'$ and conformal maps $\Phi_1^D,\Phi_2^D,\Phi_1^E,\Phi_2^E \in \cont^{1,\alpha'}$ such that
\[
	D \stackrel{\Phi_2^D}{\longmapsfrom} D' \stackrel{\Phi_1^D}{\longmapsto} B \stackrel{\Phi_1^E}{\longmapsfrom} E' \stackrel{\Phi_2^E}{\longmapsto} E
\]
and $\Phi = \Phi_2^E \circ (\Phi_1^E)^{-1} \circ \Phi_1^D \circ (\Phi_2^D)^{-1}$. The claimed mapping property is known to hold between any pair of consecutive domains in the above chain since either the relevant conformal map or its inverse is smooth enough. This proves the general claim.
\end{proof}

It follows immediately from the above theorem that $u^\mu_\bg = u^{\tilde \mu}_\bg \circ \Phi$ and $w^{\mu} = u^\mu - u_\bg^\mu = w^{\tilde \mu} \circ \Phi$ transform similarly. Also observe that point current sources are not ``deformed'' by the transformation; if $\mu = \sum_j c_j \delta_{x_j}$, then $\tilde \mu = \sum_j c_j \delta_{\Phi(x_j)}$.

\begin{lemma}\label{lemma:u1:Neumann:function}
Let $\Phi : D \to B$ be a conformal map. The representation formula \eqref{eq:u1:Neumann:function} holds and, furthermore,
\begin{equation}
\label{eq:nabla:u1:Neumann:function}
	\nabla u_\bg^\mu(x) = \int_{\partial D} \nabla_x \left(N(\Phi(x),\Phi(y)) \right) \d \mu(y)
\end{equation}
for any $x \in D$.
\end{lemma}

\begin{proof}
Let $U \csubset B$ be arbitrary. The explicit representation \eqref{eq:unit:disk:Neumann:function} shows that $\sup_{x \in U, y \in \partial B} |N(x,y)|, \sup_{x \in U, y \in \partial B} |\nabla_x N(x,y)| < \infty$ and it is known that the operator $f \mapsto \tilde u^f_\bg|_U : H^{-s}_\diamond(\partial B) \to H^1(U)/\R$, $s \in \R$, where $\tilde u_\bg^f = u_\bg^f$ solves \eqref{eq:def:u1:distributional} in $D = B$, is bounded \cite[Appendix]{convexbss}. It follows by a straightforward density argument
that
\begin{equation}
\label{eq:u1:Neumann:function:B}
	\langle N(x,\cdot), f \rangle, \qquad x \in U,
\end{equation}
defines a representative of $\tilde u_\bg^f|_U$ and
$
	\langle \nabla_x N(x,\cdot), f \rangle,\ x \in U,
$
defines $\nabla \tilde u_\bg^f|_U$. Since $U$ was arbitrary, the above expressions remain valid formulas for (a representative of) $\tilde u_\bg^f$ and $\nabla \tilde u^f$ in the whole disk $B$.

The general claim for an arbitrary piecewise $\cont^{1,\alpha}$ smooth $D$ now follows from Lemma~\ref{lemma:same:u1} %
and Theorem~\ref{thm:pull:back:var},
since
\[
	u_\bg^\mu
	= u^{\tilde \mu}_\bg \circ \Phi
	= \int_{\partial B} N(\Phi(\cdot),y) \d \tilde \mu(y)
	= \int_{\partial D} N(\Phi(\cdot),\Phi(y)) \d \mu(y)
\]
modulo constant functions.
\end{proof}

\begin{remark}
The representative given by \eqref{eq:u1:Neumann:function:B} (for $f \in \cont(\partial B)$) has zero mean on the unit circle $\int_{\partial B} \tilde u_\bg^f \d s = 0$. Correspondingly, the representative given by \eqref{eq:u1:Neumann:function} (for any $g \in \cont(\partial D)$) satisfies
\[
	\int_{\partial B} \tilde u_\bg^g \circ \Phi^{-1} \d s = 0,
\]
and its mean does not generally vanish on $\partial D$. Thus $N(\Phi(\cdot),\Phi(\cdot))$ is not the usual Neumann--Green function of $D$ (cf., e.g., \cite{bruhl2001factorization}), %
but induces a different normalization criterion (``ground level'') for $u_\bg$.
\end{remark}

The rest of this section focuses on proving the claimed properties of $\Lambda_\sigma - \Lambda_\bg$ in Theorem~\ref{thm:NtoD:continuity}.

\begin{theorem}[Factorization]\label{thm:NtoD:factorization}
Let $U \csubset D$ be an open set such that $\supp(1-\sigma) \subset U$. Then for any $\mu, \eta \in \cont'_\diamond(\partial D)$,
\begin{subequations}
\begin{align}
\label{eq:factorization:G}
	Q_\sigma(\mu,\eta)
	&= \int_U (SG \mu) \cdot (G \eta) \d x, \\
\label{eq:factorization:A}
	&= \int_{D} (1-\sigma) \nabla(A\eta) \cdot \nabla((I+P)A\mu) \d x,
\end{align}
\end{subequations}
where the operators
\begin{align*}
	G &: \cont'_\diamond(\partial D) \to (L^2(U))^2, \quad
	\mu \mapsto \nabla u_\bg^\mu|_U, \\
	A &: \cont'_\diamond(\partial D) \to H^1(U)/\R, \quad \mu \mapsto u_\bg^\mu|_U, \\
	P &: H^1(U)/\R \to H^1(D)/\R, \quad u_\bg^\mu|_U \mapsto w^\mu,
\end{align*}
and $S : (L^2(U))^2 \to (L^2(U))^2$ are linear and bounded. In addition, $G$ is continuous between the weak* topology of $ \cont'_\diamond(\partial D)$ and the strong topology of $(L^2(U))^2$. 
\end{theorem}

\begin{proof}
The continuity of the (well-defined) operators $G$ and $A$ follow from Lem\-ma~\ref{lemma:u1:Neumann:function}:
\[
	| (A \mu)(x) | = \left|\int_{\partial D} N(\Phi(x),\Phi(y)) \d \mu(y)\right|
	\leq \sup_{\tilde x \in \Phi(U),\; y \in \partial B} |N(\tilde x,y)| \cdot \| \mu \|_{\cont'_\diamond(\partial D)},
\]
which is finite, as seen from the explicit formula \eqref{eq:unit:disk:Neumann:function}. Similarly, $|(G\mu)(x)| = |\nabla_x (A \mu)(x)| \leq C \|\mu \|_{\cont'_\diamond(\partial D)}$.
Furthermore, if $\mu_j \wto \mu$, then $(G \mu_j)(x)$ converges pointwise for $x \in U$ by \eqref{eq:nabla:u1:Neumann:function} and
\[
	|(G \mu_j)(x)| \leq \sup_{x \in U,\; y \in \partial D}  |\nabla_x N(\Phi(x),\Phi(y))| \cdot \| \mu_j\|_{\cont'(\partial D)} \leq C,
\]
since since a weak* convergent sequence is strongly bounded. Therefore $|(G \mu_j)(x) - (G \mu)(x)|^2$ is uniformly bounded for $x \in U$ and it follows from the dominated convergence theorem that $G \mu_j \to G \mu$ in $(L^2(U))^2$.

Now consider the variational problem
\begin{equation*}
	\int_D \sigma \nabla w \cdot \nabla v \d x = f(v) \qquad \mbox{ for all }v \in H^1(D)/\R.
\end{equation*}
As in the proof of Lemma~\ref{lemma:umu:solution:and:w}, for any $f \in (H^1(D)/\R)'$, there exists a unique solution $w \in H^1(D)/\R$ and the mapping $T : (H^1(D)/\R)' \to H^1(D)/\R, f \mapsto w$ is continuous.

Set $F : H^1(U)/\R \to (H^1(D)/\R)'$, $g \mapsto (v \mapsto \int_U (1-\sigma) \nabla g \cdot \nabla v \d x)$ and $P = TF : H^1(U)/\R \to H^1(D)/\R$, both of which are bounded. Then
\[
	w^\mu = T[v \mapsto \textstyle\int_U (1-\sigma) \nabla u_\bg^\mu \cdot \nabla v \d x] = T F u_\bg^\mu|_U = P A \mu.
\]
It follows from definition \eqref{eq:def:relative:NtoD:and:Q} and Lemma~\ref{lemma:umu:solution:and:w} that
\[
\begin{split}
	Q_\sigma(\mu,\eta)
	&= \int_U (1 - \sigma)\nabla (u_\bg^\mu + w^\mu) \cdot \nabla u_\bg^\eta \d x \\
	&= \int_U (1 - \sigma)\nabla ((I+P) A \mu) \cdot \nabla (A \eta) \d x.
\end{split}
\]

Similarly, set $\tilde F : (L^2(U))^2 \to (H^1(D)/\R)'$, $g \mapsto (v \mapsto \int_U (1-\sigma) g \cdot \nabla v \d x)$, whence
\[
	w^\mu = T \tilde F \nabla u_\bg^\mu|_U = T \tilde F G \mu
\]
and the factorization \eqref{eq:factorization:G} holds for $S = (1-\sigma)(I + \nabla T\tilde F)$.
\end{proof}

\begin{lemma}\label{lemma:relative:NtoD:self:duality}
The bilinear form $Q_\sigma$ is symmetric.
\end{lemma}

\begin{proof}
Let $\Phi : D \to B$ be a conformal map. Due to Theorem~\ref{thm:pull:back:var} and integration by substitution (cf. equation \ref{eq:conformal:substitution:proof}), %
\begin{align*}
	Q_\sigma(\mu,\eta)
	= \int_D (1-\sigma) \nabla u^\mu \cdot \nabla u^\eta_\bg \d x
	= \int_B (1-\tilde \sigma) \nabla u^{\tilde \mu} \cdot \nabla u^{\tilde \eta}_\bg \d x
\end{align*}
where $\tilde \sigma = \sigma \circ \Phi^{-1}$
and $\tilde \mu, \tilde \eta$ are the weak Neumann boundary values of $u^\mu \circ \Phi^{-1}$ and $u_\bg^\eta \circ \Phi^{-1}$, respectively.
By Lemma~\ref{lemma:same:NtoD:definitions}, this equals
\[
	\int_B (1-\tilde \sigma) \nabla u_\bg^{\tilde \mu} \cdot \nabla u^{\tilde \eta} \d x
	= \int_D (1-\sigma) \nabla u_\bg^\mu \cdot \nabla u^\eta \d x
	= Q_\sigma(\eta,\mu)
	.
\qedhere
\]
\end{proof}

So far, we have shown that $Q_\sigma$ is well-defined and bounded. However, it remains to prove that, for each $\mu \in \cont'_\diamond(\partial D)$, there exists a unique $(\Lambda_\sigma - \Lambda_\bg)\mu \in \cont(\partial D)/\R$ such that \eqref{eq:def:relative:NtoD:and:Q} is satisfied.

\begin{lemma}[Continuity of $(\Lambda_\sigma - \Lambda_\bg) \mu$]\label{lemma:relative:NtoD:is:trace:of:w}
Let $\mu \in \cont'_\diamond(\partial D)$ and $w^\mu$ be as in \eqref{eq:def:w}. Then $(\Lambda_\sigma - \Lambda_\bg)\mu = \gamma w^\mu \in \cont(\partial D)/\R$.
\end{lemma}

\begin{proof}
Let $\Phi$, $\tilde \mu$, and $\tilde \sigma$ be as in Theorem~\ref{thm:pull:back:var}. Since $w^{\tilde \mu}$ (as in equation \ref{eq:def:w} with $D=B$ and $\sigma = \tilde \sigma$), is harmonic near the boundary $\partial B$ and satisfies a homogeneous Neumann condition, $w^{\tilde \mu}|_{\overline{B \setminus U}} \in \cont^\infty(\overline{ B \setminus U })/\R$ for some $U \csubset B$. 
By Theorem~\ref{thm:pull:back:var}, $w^\mu = w^{\tilde \mu} \circ \Phi$ where $\Phi \in \cont(\overline D)$. Therefore the trace coincides with the pointwise limit of a uniformly continuous function (equivalence class) $w^\mu|_{\overline{D \setminus \Phi^{-1}(U)}}$ on the boundary $\partial D$ and $\gamma w^\mu \in \cont(\partial D)/\R$.

For an arbitrary $f \in \cont_\diamond(\partial D) \subset H^{-1/2}_\diamond(\partial D)$,
\begin{equation*}
\begin{split}
	\int_{\partial D} \gamma w^\mu f \d s &= \langle f, \gamma w^\mu \rangle
	= \int_D \sigma \nabla w^\mu \cdot \nabla u^f \d x \\
	&= \int_D (1-\sigma) \nabla u_\bg^\mu \cdot \nabla u^f \d x 
	= Q_\sigma(f,\mu) = Q_\sigma(\mu,f).
\end{split}
\end{equation*}
It now follows from Theorem~\ref{thm:NtoD:factorization} and the weak* density of $\cont_\diamond(\partial D)$ in $\cont'_\diamond(\partial D)$ that for any $\eta \in \cont'_\diamond(\partial D)$, 
\[
	\int_D \gamma w^\mu \d \eta
	= \lim_{\eta_j \wto \eta} \int_D \gamma w^\mu \eta_j \d s
	= \lim_{\eta_j \wto \eta} Q_\sigma(\mu,\eta_j)
	= Q_\sigma(\mu,\eta).
\]
Evaluating $\int_{\partial D} [(\Lambda_\sigma-\Lambda_\bg)\mu] \d (\delta_x - \delta_y)$ for all $x,y \in \partial D$ also shows that $(\Lambda_\sigma-\Lambda_\bg)\mu : \partial D \to \R$ is unique up to a constant.
\end{proof}

The continuity of $\Lambda_\sigma - \Lambda_\bg : \cont'_\diamond(\partial D) \to \cont(\partial D)/\R$, as well as \eqref{eq:NtoD:weak:convergence}, follow from Theorem~\ref{thm:NtoD:factorization}, Lemma~\ref{lemma:relative:NtoD:self:duality}, and the boundedness of weak* convergent sequences.

\subsection{Bisweep data}

It is known by Lemma~\ref{lemma:same:NtoD:definitions} that, in the unit disk $B$, Definition~\ref{def:bisweep:data} coincides with the definition of bisweep data in \cite{bisweep}. The identity \eqref{eq:bisweep:conformal} follows directly from Theorem~\ref{thm:pull:back:var} and thus we are ready to state:

\begin{proof}[Proof of Theorem~\ref{thm:partial:calderon}]
Let $\Phi$ be a conformal map from $D$ to the unit disk $B$ and $\tilde \Gamma = \Phi(\Gamma)$. By \cite[Corollary 2.3 \& Remark 2.4]{bisweep}, the knowledge of $\tilde \varsigma(\tilde x, \tilde y)$ for all $(\tilde x, \tilde y) \in \tilde \Gamma \times \tilde \Gamma$ determines the unit disk conductivity $\tilde \sigma$ and $\sigma$ is recovered as $\tilde \sigma \circ \Phi$.
\end{proof}

The cornerstones of these results are the solvability of Calderón problem with full data \cite{astalapaivarinta2006} and the analyticity of bisweep data, which is also proven in \cite{bisweep}.
For completeness, we will give a (slightly extended) proof of this fact, which will also be utilized in the proofs of Theorems \ref{thm:partial:calderon:w} and \ref{thm:bisweep:NtoD:kernel}. Here the ``interior domain'' factorization \eqref{eq:factorization:G} is utilized instead of the transmission-problem-based factorizations of the Neumann-to-Dirichlet map in \cite{sweep}, \cite{multisweep}, \cite{bisweep} and \cite{convexbss}.

\begin{theorem}\label{thm:four:electrode:function}
Let $D = B$ be the unit disk and $\delta(\theta) := \delta_{(\cos \theta, \sin \theta)}$.
The function
\[
	\tilde \omega(\theta_1,\theta_2,\theta_3,\theta_4) = Q_\sigma(\delta(\theta_1) - \delta(\theta_2), \delta(\theta_3) - \delta(\theta_4))
\]
is (jointly) analytic in $\R^4$. Furthermore, it extends to a complex analytic function in $V^4 \subset \C^4$, where $V = \{ z \in \C : |\Im z| < l \}$ for some $l > 0$.
\end{theorem}

\begin{proof}
Consider the function
\begin{equation}
\label{eq:complex:g}
\begin{split}
	g(x,\theta,\phi)
	&:= \nabla u^{\delta(\theta) - \delta(\phi)}_\bg(x)
	= -\frac1\pi \begin{bmatrix}
		\Re\left( \frac{\xi_x - e^{i \theta}}{|\xi_x - e^{i \theta}|^2} - \frac{\xi_x - e^{i \phi}}{|\xi_x - e^{i \phi}|^2} \right) \\
		\Im\left( \frac{\xi_x - e^{i \theta}}{|\xi_x - e^{i \theta}|^2} - \frac{\xi_x - e^{i \phi}}{|\xi_x - e^{i \phi}|^2} \right)
		\end{bmatrix} \\
	&=\frac1{2\pi} \begin{bmatrix}
		-\frac{1}{\overline{\xi_x} - e^{-i \theta}} - \frac{1}{\xi_x - e^{i \theta}} + \frac{1}{\overline{\xi_x} - e^{-i \phi}} + \frac{1}{\xi_x - e^{i \phi}}  \\
		\frac{i}{\overline{\xi_x} - e^{-i \theta}} - \frac{i}{\xi_x - e^{i \theta}} - \frac{i}{\overline{\xi_x} - e^{-i \phi}} + \frac{i}{\xi_x - e^{i \phi}}
		\end{bmatrix},
\end{split}
\end{equation}
where $\xi_x = x_1 + i x_2$. The ultimate step on the first line follows from a straightforward complexification of \eqref{eq:nabla:u1:Neumann:function} in the case $\Phi = \mbox{Id}$.

Let $0 < R < 1$ be such that $\supp(1-\sigma) \csubset U := B_R(0)$. 
Set $l = -\frac12 \log R$ and let $V$ be as above.
It is clear from the explicit representation \eqref{eq:complex:g} that each component of $g(x,\cdot,\cdot)$ extends to a holomorphic function of $\theta, \phi \in V$ when $x \in U$. In particular, there exists an $m > 0$ such that $|\xi_x - e^{i \theta}| > m$ and $|\overline \xi_x - e^{-i \theta}| > m$ for all $\theta \in V$, $x \in U$. Hence the real and imaginary parts of $g(\cdot,\theta,\phi)$ are in $(L^2(U))^2$.

The real-valued operator $S$ in \eqref{eq:factorization:G} can be naturally complexified as $S(u + iv) = Su + iSv$. Thus the extension of
\begin{equation}
\label{eq:w:g}
	\tilde \omega(\theta_1,\theta_2,\theta_3,\theta_4) = \int_U g(x,\theta_1,\theta_2) \cdot (S_yg(y,\theta_3,\theta_4))(x)\d x
\end{equation}
to $V^4$ is well-defined.

It is apparent from the explicit representation \eqref{eq:complex:g} that for any $\theta, \phi \in V$ and $\eta \in \C \setminus \{ 0 \}$ such that $\theta + \eta \in V$, there exists a bound $M > 0$ so that the inequality
\begin{equation*}
	\left| \frac{g(x,\theta+\eta,\phi) - g(x,\theta,\phi)}{\eta} - \frac{\partial g}{\partial \theta}(x,\theta,\phi) \right| \leq M |\eta|
\end{equation*}
holds for all $x \in U$. 
As a result, for any bounded linear mapping $L : (L^2(U))^2 \to (L^2(U))^2$ (in particular, ${\rm Id}$ 
and $S$),
\begin{equation*}
\begin{split}
	&\left\| \frac{(L g(\cdot,\theta+\eta,\phi)) - (L g(\cdot,\theta,\phi))}{\eta} - \left(L\frac{\partial g}{\partial \theta}(\cdot,\theta,\phi)\right) \right\|_{(L^2(U))^2} \\
	&= \left\|\left| L\left( \frac{g(\cdot,\theta+\eta,\phi) - g(\cdot,\theta,\phi)}{\eta} - \frac{\partial g}{\partial \theta}(\cdot,\theta,\phi) \right) \right|\right\|_{L^2(U)}
	\leq \| L \| M \sqrt{|U|} |\eta|.
\end{split}
\end{equation*}
The same holds true for the other complex variable $\phi$ of $g$.
This means that the operator $(\theta,\phi) \mapsto Lg(\cdot,\theta,\phi) : V^2 \to (L^2(B))^2$ is strongly holomorphic in both variables. (The complex partial derivative w.r.t. $\theta$ is $(\theta,\phi) \mapsto L (\frac{\partial g}{\partial \theta}(\cdot,\theta,\phi))$.)
Therefore the expression \eqref{eq:w:g}, which is a continuous bilinear form in $(L^2(U))^2$, is separately holomorphic in $V$ with respect to each variable while the others assume arbitrary fixed values in $V$. The claim of joint holomorphy in $V^4$ follows from Hartog's Theorem \cite[Theorem~2.2.8]{Hormander73}, and the analyticity in $\R^4$ follows via restriction.
\end{proof}

By similar arguments (cf. \cite{bisweep}), one can prove the following

\begin{corollary}\label{corollary:four:electrode:disk} %
The mapping $\tilde w : (\partial B)^4 \to \R^4$,
\[
	\tilde w(z_1,z_2,z_3,z_4) = Q_\sigma(\delta_{z_1} - \delta_{z_2},\delta_{z_3} - \delta_{z_4})
\]
extends to a complex analytic function in some neighbourhood of $(\partial B)^4 \subset \C^4$.
\end{corollary}

Together with \eqref{eq:NtoD:weak:convergence}, the next lemma proves Theorem~\ref{thm:bisweep:NtoD:kernel}.

\begin{lemma}\label{lemma:bisweep:NtoD:kernel}
For any $f, g \in \cont_\diamond(\partial D)$, it holds that
\begin{equation*}
	Q_\sigma(f,g)
	= -\tfrac{1}2 \iint\limits_{\partial D \times \partial D} f(x) g(y) \varsigma(x,y) \d s(x) \d s(y).
\end{equation*}
In addition,
\begin{equation}
\label{eq:bisweep:NtoD:series}
	\varsigma_\sigma(x,y) = \lim_{N \to \infty} \sum_{i,k=1}^N \lambda_{ik} \left(\phi_k(x)\phi_i(x) - 2\phi_k(x)\phi_i(y) + \phi_k(y)\phi_i(y) \right)
\end{equation}
uniformly in $\partial D \times \partial D$, where $\{ \phi_i \}_{i=0}^\infty$ is a certain orthonormal basis for $L^2(\partial D)$ and $\lambda_{ij} = Q_\sigma(\phi_i,\phi_j)$, $i,j = 1,2,3,\ldots$
\end{lemma}

\begin{proof}
Let $c_0 \in \cont^\alpha([0,L))$ be a natural parametrization of $\partial D$, %
i.e.,
$
	\int_{c_0((a,b))} \d s = b-a
$
for all $0 \leq a \leq b < L$. Denote by $c$ the periodic continuation of $c_0$ to $\cont^\alpha(\R)$ and let \[
\{ \varphi_j(x) \}_{j=0}^\infty = \left\{ \tfrac1{\sqrt L}, \sqrt{\tfrac2L} \cos \tfrac{2\pi x}L, \sqrt{\tfrac2L} \sin \tfrac{2\pi x}L, \sqrt{\tfrac2L} \cos \tfrac{4 \pi x}L, \ldots \right\}                                                                                     \]
be the standard trigonometric Fourier basis of $L^2((0,L))$.

Let $\Phi : D \to B$ be a conformal map and $\tilde \sigma = \sigma \circ \Phi^{-1}$.
Corollary~\ref{corollary:four:electrode:disk} shows that the function
$\tilde w(\tilde x,\tilde y,\tilde p,\tilde q) = Q_{\tilde \sigma}(\delta_{\tilde x} - \delta_{\tilde y}, \delta_{\tilde p} - \delta_{\tilde q})$ is smooth (jointly analytic) on $\partial B \times \partial B \times \partial B \times \partial B$. Thus
\[
	\omega(x',y',p',q') := Q_{\tilde \sigma}(\delta_{\Phi(c(x'))} - \delta_{\Phi(c(y'))},\delta_{\Phi(c(p'))} - \delta_{\Phi(c(q'))})
\]
is Hölder continuous in $\R^4$ (because $\Phi \in \cont^\alpha(\overline D)$) and therefore the following four-dimensional Fourier series of $\omega$, where $\varphi_{ijkl}(x',y',p',q') = \varphi_i(x')\varphi_j(y')\varphi_k(p')\varphi_l(q')$,
converges uniformly~\cite{golubov1984}:%
\begin{equation*}
	\omega
	= \lim_{N \to \infty} \sum_{i,j,k,l=0}^N \varphi_{ijkl} \int\limits_{(0,L)^4} \varphi_{ijkl}\omega \d t^4.
\end{equation*}
Now set $\phi_j = \varphi_j \circ c_0^{-1}$. By Theorem~\ref{thm:pull:back:var}, Lemma~\ref{lemma:relative:NtoD:is:trace:of:w}, and definition \eqref{eq:def:four:electrode:function},
\[
	w(x,y,p,q) = \omega(c_0^{-1}(x),c_0^{-1}(y),c_0^{-1}(p),c_0^{-1}(q)) %
	= w^{p,q}(x) - w^{p,q}(y),
\]
where $w^{p,q} \in \cont_\diamond(\partial D)$ denotes the continuous representative of the Dirichlet trace $\gamma w^{\delta_p-\delta_q}$ with zero mean and $w^{\delta_p-\delta_q}$ is as in \eqref{eq:def:w}.

Similarly, let $w_i$ denote the representative of $\gamma w^{\phi_i}$ with zero mean. 
By Theorem~\ref{thm:NtoD:continuity} and Lemma~\ref{lemma:relative:NtoD:is:trace:of:w}, $\int_{\partial D} w^{p,q} \phi_i \d s = w_i(p)-w_i(q)$ and $\lambda_{ij} = \lambda_{ji} = \int_{\partial D} w_i \phi_j \d s$ for all $i,j = 0,1,2,\ldots$ if one sets $\lambda_{0i} = \lambda_{j0} = 0$.
Define $\phi_{ij}, \phi_{ijk}, \phi_{ijkl}$ like $\varphi_{ijkl}$ above. Then
{\allowdisplaybreaks %
\begin{align*}
	&\int_{(0,L)^4} \varphi_{ijkl} \omega \d t^4 
	= \int_{(\partial D)^4} \phi_{ijkl} w \d s^4 \\
	&= \int_{(\partial D)^3} \phi_{jkl} \int_{\partial D} (w^{p,q}(x) - w^{p,q}(y))\phi_i(x)\d s(x) \d s^3 \\
	&= \int_{(\partial D)^3} \phi_{jkl} \left(\int_{\partial D} w^{p,q} \phi_i \d s  - \delta_{i0} w^{p,q}(y)L^{\frac12} \right) \d s^3 \\
	&= \int_{(\partial D)^3} \phi_{jkl} \left((w_i(p)-w_i(q)) - \delta_{i0} w^{p,q}(y)L^{\frac12} \right) \d s^3 \\
	&= \int_{(\partial D)^2} \phi_{kl} \left(\delta_{j0} (w_i(p)-w_i(q)) - \delta_{i0} (w_j(p)-w_j(q)) \right) L^{\frac12} \d s^2 \\
	&= \int_{\partial D} \phi_l \left(\delta_{j0} (\lambda_{ik} - \delta_{k0} L^{\frac12} w_i(q)) - \delta_{i0} (\lambda_{jk} - \delta_{k0} L^{\frac12} w_j(q)) \right) L^{\frac12} \d s\\
	&= \left(\delta_{j0} (\delta_{l0}\lambda_{ik} - \delta_{k0} \lambda_{il}) - \delta_{i0} (\delta_{l0}\lambda_{jk} - \delta_{k0} \lambda_{jl}) \right) L,
\end{align*}}%
where $\delta_{ij}$ is the Kronecker delta.

Inserting the above expression into
\[
	\varsigma_\sigma(x,y) = w(x,y,x,y) = \lim_{N \to \infty} \sum_{i,j,k,l=0}^N \phi_i(x)\phi_j(y)\phi_k(x)\phi_l(y) \int\limits_{(0,L)^4} \varphi_{ijkl}\omega \d t^4
\]
and simplifying yields \eqref{eq:bisweep:NtoD:series} and the limit is uniform in $\partial D \times \partial D$.

As a result, 
{\allowdisplaybreaks
\begin{align*}
	&\iint\limits_{\partial D \times \partial D} \varsigma_\sigma(y,z) f(y)g(z) \d s(y) \d s(z)
	= -2 \lim_{N \to \infty} \sum_{i,j=1}^N {\lambda}_{ij} \int_{\partial D} \phi_i f \d s \int_{\partial D} \phi_j g \d s \\
	&= -2 \lim_{N \to \infty}  \int_{\partial D} \left(\textstyle\sum_{i=1}^N \phi_i \textstyle \int_{\partial D} \phi_i f \d s \right) \left((\Lambda_\sigma - \Lambda_\bg)\left(\textstyle\sum_{j=1}^N \phi_j \textstyle \int_{\partial D} \phi_j g \d s \right)\right) \d s \\
	&= -2 \int_{\partial D} f \left((\Lambda_\sigma - \Lambda_\bg) g \right) \d s,
\end{align*}}%
where the last equality follows from the boundedness of $\Lambda_\sigma - \Lambda_\bg : L^2_\diamond(\partial D) \to L^2(\partial D)/\R$ and the fact that $\sum_{i=1}^N \phi_i \int_{\partial D} \phi_i f \d s \rightarrow f$ in $L^2(\partial D)$ for any $f \in L^2_\diamond(\partial D)$.
\end{proof}

The analyticity of the ``four-electrode function'' $\tilde \omega$ can be combined with the conformal mapping property of point current patterns, which yields the new partial data result:

\begin{proof}[Proof of Theorem~\ref{thm:partial:calderon:w}]
Let $\Phi : D \to B$ be a conformal map. For an arbitrary branch of the complex logarithm, set $\tilde \Gamma = -i \log \Phi(\Gamma) \subset \R$, $\tilde \Xi = -i \log \Phi(\Xi) \subset \R$. By Theorem~\ref{thm:pull:back:var}, the measurements \eqref{eq:def:four:electrode:function} determine the quantity
$
	\tilde \omega(\theta,\theta_0, \phi,\phi_0),
$
defined in Theorem~\ref{thm:four:electrode:function},
for all $\theta,\theta_0 \in \tilde \Gamma$, $\phi,\phi_0 \in \tilde \Xi$.

Let $\theta_0 \in \tilde \Gamma$, $\phi_0 \in \tilde \Xi$ be arbitrary and $\theta_1$, $\phi_1$ be some accumulation points of $\tilde \Gamma$ and $\tilde \Xi$, respectively. As in \cite{bisweep}, for any integers $j, k \geq 1$, it holds that
\[
	\frac{\partial^j}{\partial \theta_1^j} \frac{\partial^k}{\partial \phi_1^k}  \tilde \omega(\theta_1,\theta_0,\phi_1,\phi_0)
	=
	\langle \delta_{\theta_1}^{(j)}, (\Lambda_{\tilde \sigma} - \Lambda_\bg) \delta_{\phi_1}^{(k)} \rangle,
\]
where $\delta_\theta^{(j)} : f \mapsto \frac{\d^j}{\d \theta^j} f(e^{i \theta})$ and $\Lambda_{\tilde \sigma} - \Lambda_\bg : H^{-s}_\diamond(\partial B) \to H^s(\partial B)/\R$ is as in Lemma~\ref{lemma:same:NtoD:definitions} for some sufficiently large $s$. Furthermore, the values of the partial derivatives on the left hand side are revealed by the measurements.

Due to the analyticity of $(\Lambda_{\tilde \sigma} - \Lambda_\bg) \delta_{\phi_1}^{(k)} \in \cont^\infty(\partial B)/\R$, this also determines $\langle \delta_{\phi_1}^{(j)}, (\Lambda_{\tilde \sigma} - \Lambda_\bg) \delta_{\phi_1}^{(k)} \rangle$ for all $j,k \geq 1$. By \cite[Theorem 2.2]{bisweep}, Theorem~\ref{thm:bisweep:NtoD:kernel} and \eqref{eq:bisweep:conformal}, this yields $\Lambda_\sigma - \Lambda_\bg : H^{-1/2}_\diamond(\partial D) \to H^{1/2}(\partial D)/\R$. Since $\Lambda_\bg$ is known \emph{a priori}, one recovers $\sigma \in L^\infty(D)$ from $\Lambda_\sigma$ by \cite{astalapaivarinta2006}.
\end{proof}

\begin{lemma}[Approximation of measures by continuous functions]
\label{lemma:approximation}
Let $\mu \in \cont'_\diamond(\partial D)$ be arbitrary and $V$ any relatively open subset of $\partial D$ such that $\supp \mu \csubset V$. Then there exists a sequence $\{ \mu_j \} \subset \cont_\diamond(V)$ such that $\mu_j \wto \mu$.
\end{lemma}

\begin{proof}
Using a parametrization of $\partial D$ and convolution with mollifiers in $\R$ (cf., e.g., \cite[\S1.1.5]{mazja1985}\cite[§2.1.3]{necas}), one can construct a sequence $\{ \tilde \mu_j \} \subset \cont(\partial D)$ such that $\tilde \mu_j \wto \mu$. Let $\alpha_V \in \cont(V)$ %
be such that $\supp (\alpha_V - 1) \cap \supp \mu = \emptyset$ and set $C_j = -\int_{\partial D} \alpha_V \tilde \mu_j \d s / \int_{\partial D} \alpha_V \d s$. Then
$
	\mu_j = (\tilde \mu_j + C_j)\alpha_V
$
has the claimed properties. 
\end{proof}

The measurements \eqref{eq:def:four:electrode:function} can be interpreted as follows:
Let $x,y,p,q \in \partial D$ be arbitrary. There exists sequences $\{ f_j \}, \{ g_j \} \subset \cont_\diamond(\partial D)$ such that $f_j \wto \delta_{x} - \delta_{y}$, $g_j \wto \delta_p - \delta_q$. Due to \eqref{eq:NtoD:weak:convergence},
\begin{equation*}
	w(x, y, p, q) = \lim_{j \to \infty} \langle f_j, (\Lambda_\sigma - \Lambda_\bg) g_j \rangle,
\end{equation*}
where
\[
\langle f_j, (\Lambda_\sigma - \Lambda_\bg) g_j \rangle =
	\langle f_j, \Lambda_\sigma g_j \rangle - \langle f_j, \Lambda_\bg g_j \rangle,
\]
and $\langle f_j, \Lambda_\bg g_j \rangle$ does not depend on $\sigma$ but only the shape of $D$ and is thus known \emph{a priori}.
Notice also that, by Lemma~\ref{lemma:approximation}, the sequences may be chosen so that $\{ f_j \} \subset \cont_\diamond(V)$, $\{ g_j \} \subset \cont_\diamond(W)$, which yields Corollary~\ref{corollary:L2:partial:calderon}.

\section{Reconstruction of $\sigma$ from bisweep data}\label{sec:rec:numerics}

\subsection{Reconstruction method}

Given a conformal map $\Phi : D \to B$ one can compute a reconstruction of the unit disk conductivity $\tilde \sigma = \sigma \circ \Phi^{-1}$ from the bisweep data $\varsigma_{\tilde \sigma}$ and map it back to $D$ using $\Phi^{-1}$.
Due to \eqref{eq:bisweep:conformal}, $\varsigma_{\tilde \sigma}$ can be computed from $\varsigma_\sigma$. In addition, Theorem~\ref{thm:bisweep:NtoD:kernel} provides means for converting between bisweep data and the relative Neumann-to-Dirichlet map $\Lambda_\sigma - \Lambda_\bg : L^2_\diamond(D) \to L^2(D)/\R$.
However, there are some subtleties concerning discretization and relationship to electrode measurements in this approach.

We consider the setting where $\Phi : D \to B$ is a conformal map, $\tilde x_1,\ldots,\tilde x_n \in \partial B$ equispaced points, and the bisweep data $\varsigma_{ij} = \varsigma_\sigma(x_i,x_j)$ are given for all $i,j=1,\ldots,n$, where $x_j := \Phi^{-1}(\tilde x_j)$.
From a practical point of view, this approximates a measurement where one can select the positions $x_1,\ldots,x_n$ of $n$ electrodes on the boundary $\partial D$ of the object according to a pre-defined pattern, and then conduct the (relative) EIT measurements (cf. \cite{pem}). However, unlike in smooth domains, in piecewise smooth $D$, this relationship between relative point electrode measurements and the realistic complete electrode model \cite{eit99} has not been rigorously studied.

By \eqref{eq:bisweep:conformal}, the discrete bisweep data $\varsigma_{ij}$ correspond to the unit disk data $\varsigma_{\tilde \sigma}$ sampled on the regular grid $\{\tilde x_1,\ldots,\tilde x_n\} \times \{\tilde x_1,\ldots,\tilde x_n\}$ and, due to Theorem~\ref{thm:bisweep:NtoD:kernel},
their discrete two-dimensional Fourier coefficients
\[
	c_{ij} := \sum_{k,l=1}^n \frac{4 \pi^2}{n^2} \phi_i(\tilde x_l) \phi_j(\tilde x_k) \varsigma_{kl}, \qquad 1 \leq i,j \leq n-1
\] 
with respect to (for example) the basis
\begin{equation*}
	\{ \phi_i((\cos \theta,\sin \theta)) \}_{i=0}^{n-1}
	= \left\{ \tfrac1{\sqrt{2\pi}}, \sqrt{\tfrac1\pi} \cos \theta, \sqrt{\tfrac1\pi} \sin \theta, \sqrt{\tfrac1\pi} \cos (2\theta), \ldots \right\}
\end{equation*}
can be used to approximate
\begin{equation}
\label{eq:tilde:lambda:ij}
	\tilde \lambda_{ij}
	= Q_{\tilde \sigma}(\phi_i,\phi_j)
	= \langle \phi_j, (\Lambda_{\tilde \sigma} - \Lambda_\bg)\phi_i \rangle
\end{equation}
as $\tilde \lambda_{ij} \approx -\frac12 c_{ij}$.
Now one can compute a reconstruction of the conductivity $\sigma$ in arbitrary points $y_1,\ldots,y_m \in D$ by reconstructing $\tilde \sigma$ in $\Phi(y_1),\ldots,\Phi(y_m)$ from the approximate matrix representation $[\tilde \lambda_{ij}]_{i,j=1,\ldots,n-1}$ of the relative Neumann-to-Dirichlet map for the unit disk conductivity $\tilde \sigma$.
In the following numerical examples, the unit disk reconstruction is done using the factorization method (see Section~\ref{sec:factorization:method} below),
which aims to recover the support of the \emph{inclusions} $1-\tilde \sigma$.

\begin{remark}
If the discrete measurement on $\partial D$ does not correspond to bisweep data but is (interpreted as), for example, $\lambda_{ij}$, $i,j=1,\ldots,n-1$ (cf. eq. \ref{eq:tilde:lambda:ij}), then one could compute an approximation of $\varsigma_{ij}$ by truncating the series \eqref{eq:bisweep:NtoD:series}.
If the data can be interpreted as a general measurement with $n$ point electrodes, that is,
\[
	m_{ij} = Q_\sigma \left( \sum_{k=1}^n \alpha_{ik} \delta_{x_k}, \sum_{l=1}^n \beta_{jl} \delta_{x_l}\right),
\]
where $\alpha_1,\ldots,\alpha_{n-1} \in \R^n$ and $\beta_1,\ldots,\beta_{n-1} \in \R^n$ are linearly independent sets of mean-free vectors, then one can compute $\varsigma_{ij}$ from $m_{ij}$ by solving a system of linear equations.
\end{remark}

\subsection{Factorization method}\label{sec:factorization:method}

In this section,
a numerical method for locating the inhomogeneities in $\tilde \sigma \in L^\infty(B)$ (similar to that in, e.g., \cite{harrach09freq} and \cite{lechleiter2008factorization}) is presented. A theoretical result behind the method---originally from Kirsch \cite{kirsch1998characterization}---is stated for completeness, but all further analysis is omitted. The applicability is demonstrated by the numerical examples in Section~\ref{sec:numerical:examples}. For proofs and other properties of the factorization method, refer to, e.g.,
 \cite{bruhl2001factorization,bruhl2000numerical,hanke2003recent,hyvonengebauer2007,lechleiter2008factorization,harrach09freq}.

Assume $\Omega \csubset B$ is such that $B \setminus \overline \Omega$ is connected and $\tilde \sigma \in L^\infty(B)$ satisfies
\begin{equation}
\label{eq:positive:sigma}
	0 < c \leq \tilde \sigma \leq C < 1 \ \mbox{ a.e. in }\Omega,
	\qquad
	\tilde \sigma = 1 \ \mbox{ in }B \setminus \overline \Omega.
\end{equation}
Let $\{ v_k \} \subset L^2_\diamond(\partial B)$ be orthonormal eigenfunctions and $\{ \lambda_k \} \subset \R$ the corresponding eigenvalues of the compact operator $\Lambda_{\tilde \sigma} - \Lambda_\bg : L^2_\diamond(\partial B) \to L^2(\partial B)/\R$: %
\begin{equation*}
	(\Lambda_{\tilde \sigma} - \Lambda_\bg) v_k = \lambda_k v_k + \R, \qquad k=1,2,3,\ldots
\end{equation*} 
Denote $g_{z,d} : \partial B \to \R$, \cite{bruhl2001factorization} %
\begin{equation*}
	g_{z,d}(x) = d \cdot \nabla_z N(z,x) = \frac1\pi \frac{(z-x)\cdot d}{|z-x|^2},
\end{equation*}
where $z \in B$, $d \in \partial B \subset \R^2$ and $N$ is given by \eqref{eq:unit:disk:Neumann:function}.

\begin{theorem}[Factorization method] %
Let $d \in \partial B$ be arbitrary. For any $z \in B \setminus \partial \Omega$, $z \in \Omega$ if and only if \cite{hyvonengebauer2007}
\begin{equation}
\label{eq:picard:criterion}
	f_d(z) := \frac1{\|g_{z,d}\|_{L^2(\partial B)}^2} \sum_{k=1}^\infty \frac{\left|\int_{\partial B} g_{z,d} v_k \d s\right|^2}{|\lambda_k|} < \infty.
\end{equation}
\end{theorem}
This is applied as the following algorithm for reconstructing the \emph{inclusion} $\Omega$ from the (approximate) measurements \eqref{eq:tilde:lambda:ij}.
First, choose a reconstruction order $M < n$ and compute the singular value decomposition 
\begin{equation*}
	\mb L \mb v_k = \sigma_k \mb u_k, \qquad \mb L^T \mb u_k = \sigma_k \mb v_k, \qquad k=1,\ldots,M
\end{equation*}
of the $M \times M$ matrix $\mb L = [\tilde \lambda_{ij}]_{i,j=1}^M$ %
that depicts the relative Neumann-to-Dirichlet map $\Lambda_\sigma - \Lambda_\bg$ in a finite trigonometric basis.
The following truncation
\begin{equation*}
	\tilde f_d(z) = \frac{\sum_{k=1}^M |\mb g^T \mb v_k|^2/|\sigma_k|}{\sum_{k=1}^M|\mb g^T \mb v_k|^2},
\end{equation*} 
where $\mb g = [g_1,\ldots,g_M]$ is a vector of Fourier coefficients of $g_{z,d}$ in the same basis, is used to ``approximate'' \eqref{eq:picard:criterion}.
The reconstruction is given by $\tilde \Omega = \{ z \in B \,:\, {\rm Ind}(z) < C_\infty \}$ 
for some cut-off value $C_\infty > 0$ of the function %
\begin{equation}
\label{eq:infinity:indicator}
	{\rm Ind}(z) = \sum_{k=1}^{N_d} \tilde f_{d_k}(z)/N_d
\end{equation} 
where some (odd) number $N_d$ of equispaced points $d_k$, $k=1,\ldots,N_d$ are used as the dipole directions to reduce artefacts (cf. \cite{lechleiter2008factorization}).
In practice, ${\rm Ind}$ is sampled on some finite grid (and with different values of $C_\infty$) to produce an image of $\Omega$.

\subsection{Numerical examples}\label{sec:numerical:examples}

In the following examples, the domain $D$ is a polygon, which contains an inhomogeneity defined by
\[
	\sigma(x) =
	\begin{cases}
		\kappa_i & \mbox{ if }x \in \Omega_i \\
		1 & \mbox{ otherwise},
	\end{cases}
\]
where $0 < \kappa_i < 1$ and the simply connected inclusions $\Omega_j \csubset D$, $i=1,\ldots,N_{\rm inc}$ are strictly separated, whence \eqref{eq:feasible:sigma} and \eqref{eq:positive:sigma} are satisfied.

Approximate bisweep data are computed numerically from \eqref{eq:u1:Neumann:function} and \eqref{eq:factorization:A}, where the operators $A$ and $P$ have been discretized using the finite element method (FEM).
In more detail, the simulation is carried out as follows:

\begin{enumerate}
\item 
Construct a triangular finite element mesh (e.g., Figure~\ref{fig:mesh}) with nodes $z_1,\ldots,z_N \in \overline D$, and a corresponding piecewise linear basis $\{ \varphi_j \}_{j=1}^N \subset \cont(\overline D)$
such that each basis function $\varphi_j$ is supported on the triangles that are adjacent to mesh node $z_j$. Define
\[
	I_\Omega = \{ j \,:\, \supp(\varphi_j) \cap \Omega \neq \emptyset \} \subset \{ 1,\ldots,N \},
\]
and $U = \cup_{j \in I_\Omega} \supp(\varphi_j)$.
Also compute the \emph{stiffness matrices}
\[
	(\mb K_\sigma)_{ij} = \int_D \sigma \nabla \varphi_i(x) \cdot \nabla \varphi_j(x) \d x,
	\quad i,j = 1,\ldots,N
\]
and $\mb K_{1-\sigma} \in \R^{N\times N}$.

\item For all $x,y \in \{ x_1,\ldots,x_n \} \subset \partial D$, compute
\[
	u_j^{x,y} = A(\delta_x - \delta_y)(z_j) = N(\Phi(z_j),\Phi(x)) - N(\Phi(z_j),\Phi(y))
\]
for each $j \in I_\Omega$ and set $u_j^{x,y} = 0$ for all $j \notin I_\Omega$. This defines a finite element approximation $\hat u^{x,y} = \sum_{i=1}^N u_j^{x,y} \varphi_j$ of $A(\delta_x - \delta_y)$. Denote $\mb u_{x,y} = [u_1^{x,y},\ldots,u_N^{x,y}]^T$.

\item 
A finite element approximation $\hat w^{x,y} = \sum_{j=1}^N w_j^{x,y} \varphi_j$ of %
$P \hat u^{x,y}$ can be solved from the linear system
\[
	\mb K_\sigma \mb w_{x,y} = \mb K_{1-\sigma} \mb u_{x,y}
\]
and $\varsigma_{ij} \approx \varsigma(x,y)$, $x = x_i$, $y = x_j$ can be computed from \eqref{eq:factorization:A} as
\[
	\varsigma_{ij} = \mb u_{x,y}^T \mb K_{1-\sigma} (\mb u_{x,y} + \mb w_{x,y}).
\]
\end{enumerate}
In the examples below, these computations were done in MATLAB and the conformal map $\Phi$ was computed using the Schwartz--Christoffel toolbox \cite{driscoll1996}. %

\begin{remark}
It is also possible to simulate the data by first computing $\tilde \sigma = \sigma \circ \Phi^{-1}$, then computing $\varsigma_{\tilde \sigma}$ as described in, for example, \cite{sweep} and finally mapping the result back to $D$ in order to obtain $\varsigma_\sigma$. This approach, which is also the theoretical basis of the reconstruction method, is not taken in order to avoid an obvious \emph{inverse crime} \cite{colton1998inverse}.
In the chosen alternative procedure, the numerical conformal map $\Phi$ that is used for reconstruction is employed only in the simulation of $u_\bg|_U$.
\end{remark}

In the following numerical examples, the domain $D$ is a non-convex polygon depicted in Figures \ref{fig:noiseless} and \ref{fig:noisy}. There are two inclusions, a disk and a rectangle, which have the conductivity $\kappa_1 = \kappa_2 = \frac12$. A conformal map $\Phi : D \to B$ is constructed using the Schwartz-Christoffel toolbox and $n$ point electrodes are positioned as $x_j = \Phi^{-1}(\tilde x_j)$, where $\tilde x_j$ are equispaced on $\partial B$. We use $N_d = 15$ dipole directions for \eqref{eq:infinity:indicator} in both examples.

\begin{example}
In the first simulation, a high number $n = 128$ of point electrodes is used.
The positions of the inclusions and the electrodes are shown in Figure~\ref{fig:noiseless:reconstruction}. The unit disk
phantom $\tilde \sigma = \sigma \circ \Phi^{-1} \in L^\infty(B)$ and $\tilde x_j$ are depicted in Figure~\ref{fig:pullback}.
Discrete bisweep data $\varsigma_{ij} = \varsigma_\sigma(x_i,x_j)$, $i,j = 1,\ldots,n$ (shown in Figure~\ref{fig:bisweep}), are computed as described above. 

No artificial noise is added to the simulation and the reconstruction is conducted with a high order of $M = 64$ singular values. The result is also shown in Figure~\ref{fig:noiseless:reconstruction}, where the different shades of gray correspond to different cut-off levels $C_\infty$ (logarithmic spacing).
This suggests that in an ideal setting, the method works as desired and the supports of the inclusions are recovered well.
\end{example}

\begin{example}
In the second example, the number of electrodes is reduced to $n=16$ and artificial errors are generated as follows: The vertices of the domain are perturbed slightly, which yields a new domain $D'$. Perturbed electrode positions $x_j' \in \partial D'$ are computed by selecting the closest point to $x_j$ on $\partial D'$ and adding a small perturbation. Bisweep data $\varsigma_{ij}' = \varsigma_{\sigma'}(x_i',x_j')$ are computed and extra noise (normally distributed with standard deviation of $\frac2{100} \max |\varsigma_{ij}'|$) is added to simulate measurement error. The conductivity $\sigma' \in L^\infty(D')$ is defined so that $\sigma' = \sigma$ in $D' \cap D$ and $\sigma' = 1$ in $D' \setminus D$.

Reconstruction is then carried out as above but in the incorrect (or ideal) domain $D$ and for the incorrect electrode positions $x_j$, as if $\varsigma_{ij}' = \varsigma_{ij}$.
This simulates the effect of slight misplacement of electrodes and error in modeling the domain (while, however, still having true relative data). A lower order of $M = 12$ singular values is used.

Figures~\ref{fig:noisy:1}--\ref{fig:noisy:last} show the reconstructions from five
different noisy samples.
Geometry error and electrode misplacement are illustrated as described under Figure~\ref{fig:noisy:1}.
The effect of these errors, and the extra noise level, in data is visualized in Figure~\ref{fig:noisy:sweep}, which matches the noisy data corresponding to 
Figure~\ref{fig:noisy:1}.
This is done with the aid of \emph{sweep data} \cite{sweep}, a restriction of bisweep data where one variable is fixed.

In Figure~\ref{fig:noisy:sweep}, the dashed line is the sweep data $\varsigma_{\sigma}(x_1,\Phi_{D}^{-1}(e^{i\theta}))$, $\theta \in [0,2\pi)$ of the ideal domain $D$, and it corresponds to the data on a certain vertical (or horizontal) line approximately in the middle of Figure~\ref{fig:bisweep}. The solid line is $\varsigma_{\sigma'}(x_1',\Phi_{D'}^{-1}(e^{i\theta}))$ and it illustrates the effect of  the geometry error. On the solid line, the values of $\theta$ corresponding to the actual electrode positions $x_j'$ are marked and the matching values of $\varsigma_{\sigma'}$ are the discrete data without extra noise. In addition, the discrete noisy data $\varsigma'_{1j}$ are plotted at values of $\theta$ corresponding to the ideal electrode positions $x_j$.

Approximate inclusion locations are recovered. Also notice the packing of electrodes near the non-convex corner, which poses an obvious problem for the applicability of this method in practice.
\end{example}

\section{Discussion and concluding remarks}

The concepts of point electrodes \cite{pem} and, in particular, bisweep data \cite{bisweep}, were generalized from smooth to piecewise $\cont^{1,\alpha}$ smooth plane domains. This was achieved by introducing a generalized relative Neumann-to-Dirichlet map, whose properties were studied with the help of conformal maps, Laplace equation with measure Neumann boundary values \cite{kral1980}, and distributional Neumann problems in smooth domains \cite{bisweep}\cite{convexbss}.

New partial data results for Calderón problem were obtained in Theorem~\ref{thm:partial:calderon:w} and Corollary~\ref{corollary:L2:partial:calderon}, based on \cite{bisweep} and the full data result \cite{astalapaivarinta2006}.
Instead of assuming a priori (interior) smoothness from the conductivity $\sigma$ (as in \cite{imanuvilov2010}), boundary homogeneity \eqref{eq:sigma:boundary:homogeneity} is assumed.
In some applications, %
\eqref{eq:sigma:boundary:homogeneity} may be well-founded. For instance, if homogeneous medium is screened for hidden defects, it might be reasonable to assume boundary homogeneity, but \emph{a priori} assumption on the smoothness of the defects (such as cracks or air bubbles in concrete) could be unrealistic.

The regularity assumptions for the considered plane domains $D$ is due to Theorem~\ref{thm:pull:back:var}. The proof technique based on Theorem~\ref{thm:conformal:sharp:angle:intermediary} and weak solutions $u$ in $W^{1,1}$ appears to fail in less regular domains. The other theorems seem to remain valid in general Lipschitz domains if the weak conductivity equation is defined appropriately.
In this paper, only isotropic (i.e., scalar) conductivities were studied, but there seems to be no reason why the results presented here would not have useful counterparts in the anisotropic case too. %

It was also demonstrated how point-electrode-based methods could be used in numerical EIT, and the notion of bisweep data enables applying unit-disk based algorithms in piecewise smooth (polygonal) domains. In this paper, only the factorization method was considered, but the same approach could be applied to other methods as well. It should be noted that there are also other means of applying the factorization method in non-smooth domains, for example, \cite{lechleiter2008factorization}.

The applicability of these point-electrode-based numerical methods depends on the availability of \emph{relative data} \cite{pem},
or very accurate knowledge of the background map $\Lambda_\bg$, both of which are questionable assumptions in practice. %
However, theoretically less-studied concepts such as \emph{frequency-difference measurements} \cite{harrach09freq} 
could be used to overcome the problem. This involves analysis of the conductivity equation with complex $\sigma$ and such counterparts of the results presented in this paper are left for future studies.

\subsection*{Acknowledgements} I would like to thank Nuutti Hyvönen for his help in proofreading this paper and suggesting improvements. I am also grateful to Tri Quach and Juha Kinnunen for useful discussions.
In addition, I would like to thank \hbox{Gunther} Uhlmann for valuable feedback.

\clearpage

\def\defaultFF{0.45\textwidth} %
\def\defaultFFF{0.28\textwidth} %
\def\subcaptionwidthFF{0.65\textwidth} %
\def\subfigvsep{0em} %
\def\fullFF{0.5\textwidth} %
\def\fullFFF{0.3333\textwidth} %

\newcommand{\incsubfigFF}[3][\defaultFF]{%
\makebox[0.5\textwidth]{\subfloat[{#3}]{%
\makebox[\subcaptionwidthFF]{\includegraphics[width=#1]{#2}}}}%
}

\newcommand{\incsubfigFFF}[3][\defaultFFF]{%
\makebox[0.3333\textwidth]{\subfloat[{#3}]{%
\includegraphics[width=#1]{#2}}}%
}

\begin{figure} 
\centering
\vspace*{2cm}%
\incsubfigFF{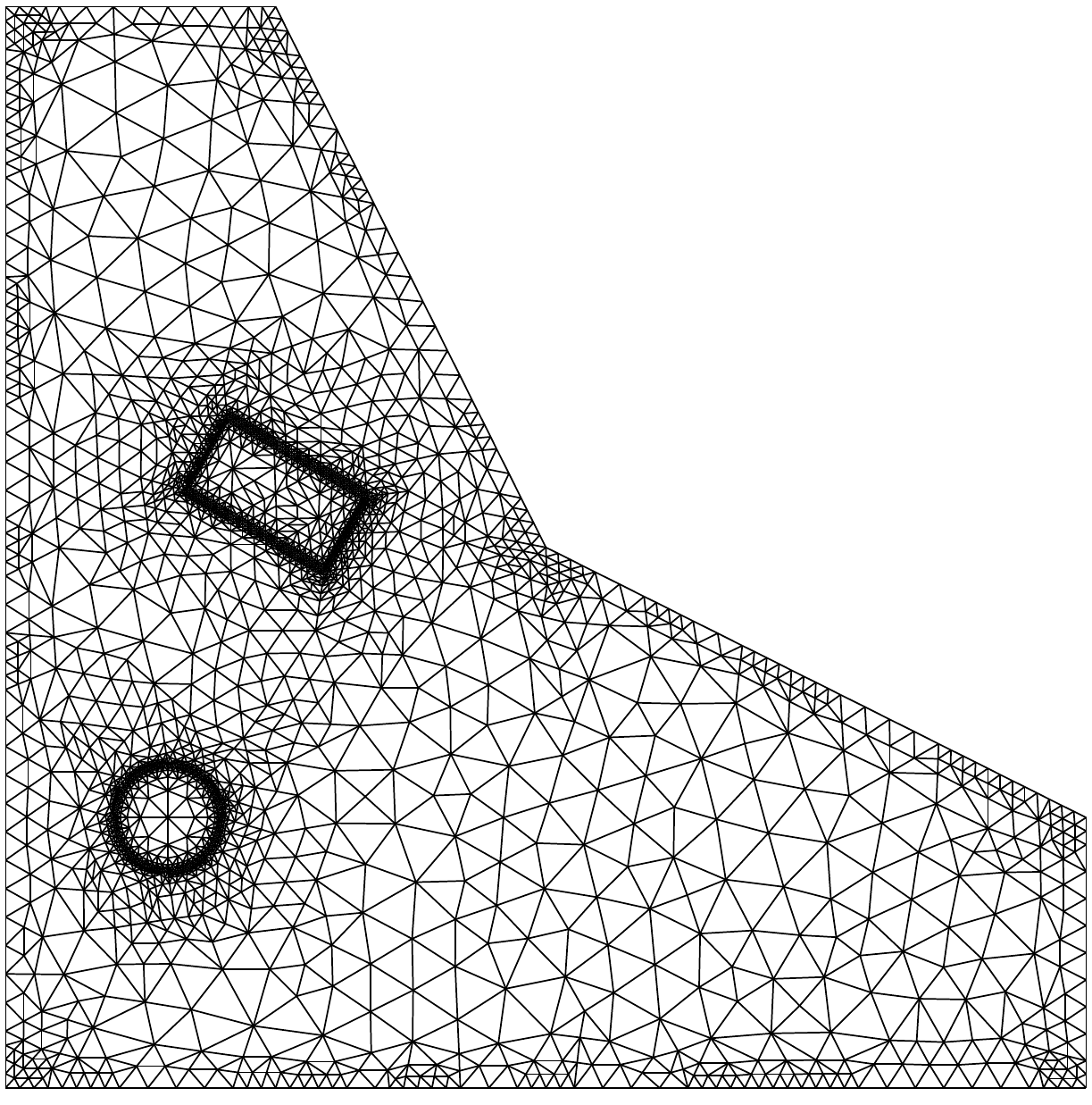}{%
\label{fig:mesh}
	The finite element mesh used in the simulation of bisweep data.
}%
\incsubfigFF{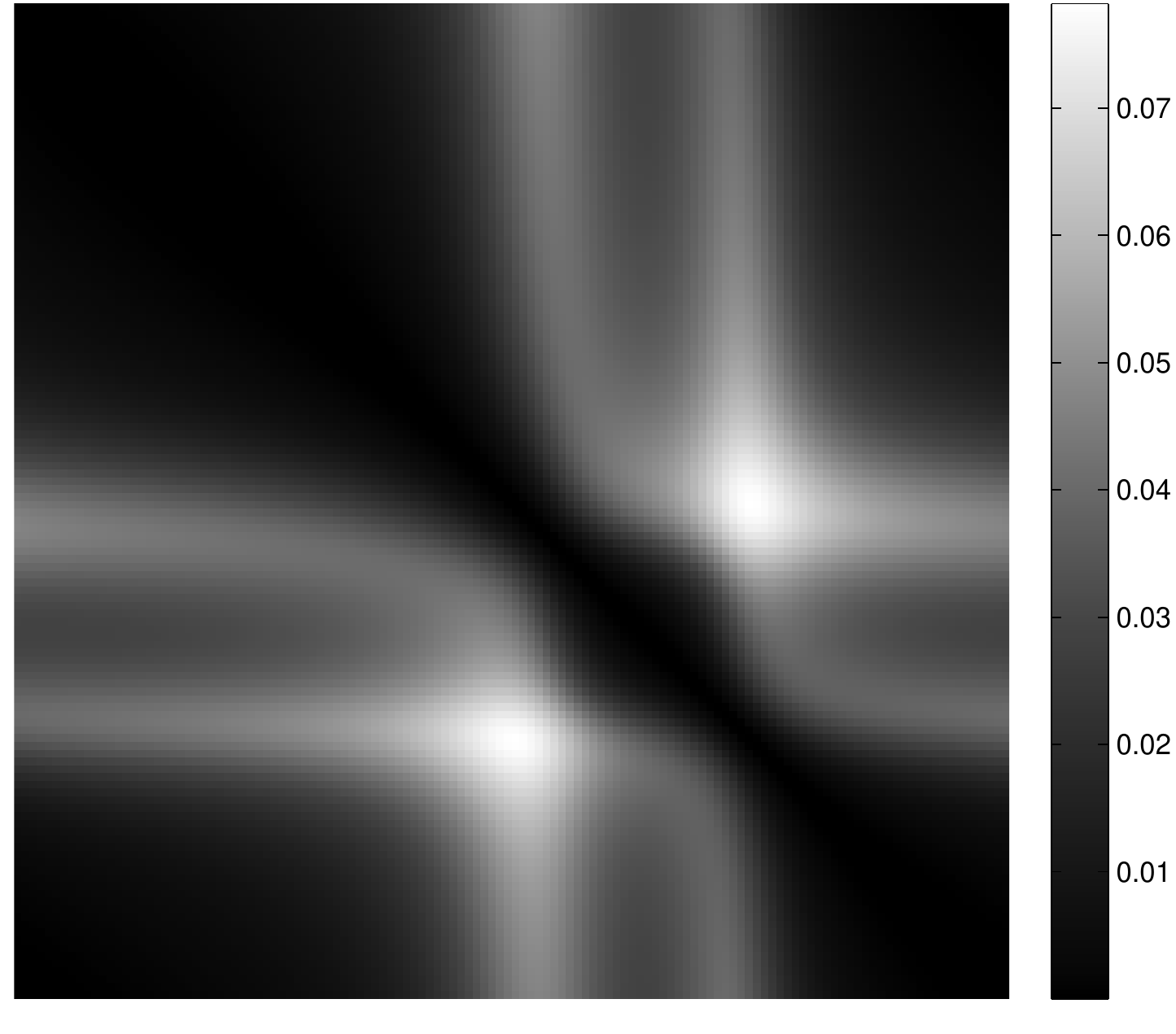}{%
\label{fig:bisweep}
	Discrete bisweep data $\varsigma_{ij}$, $i,j = 1,\ldots,128$.
}%
\\[2em]%
\incsubfigFF{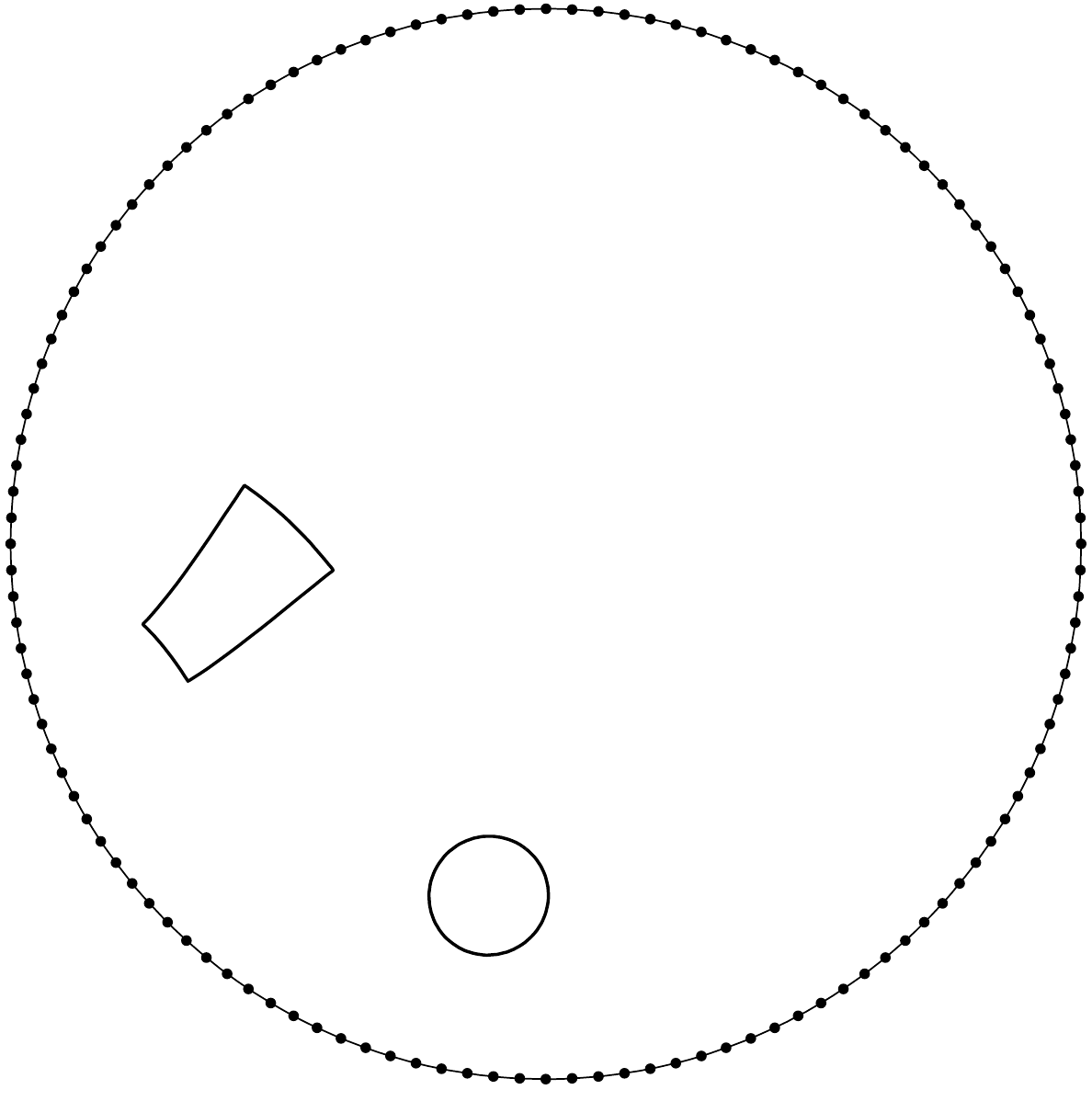}{%
\label{fig:pullback}
	Unit disk phantom $\tilde \sigma = \sigma \circ \Phi^{-1}$ and electrodes $\tilde x_1,\ldots,\tilde x_{128}$.
}%
\incsubfigFF{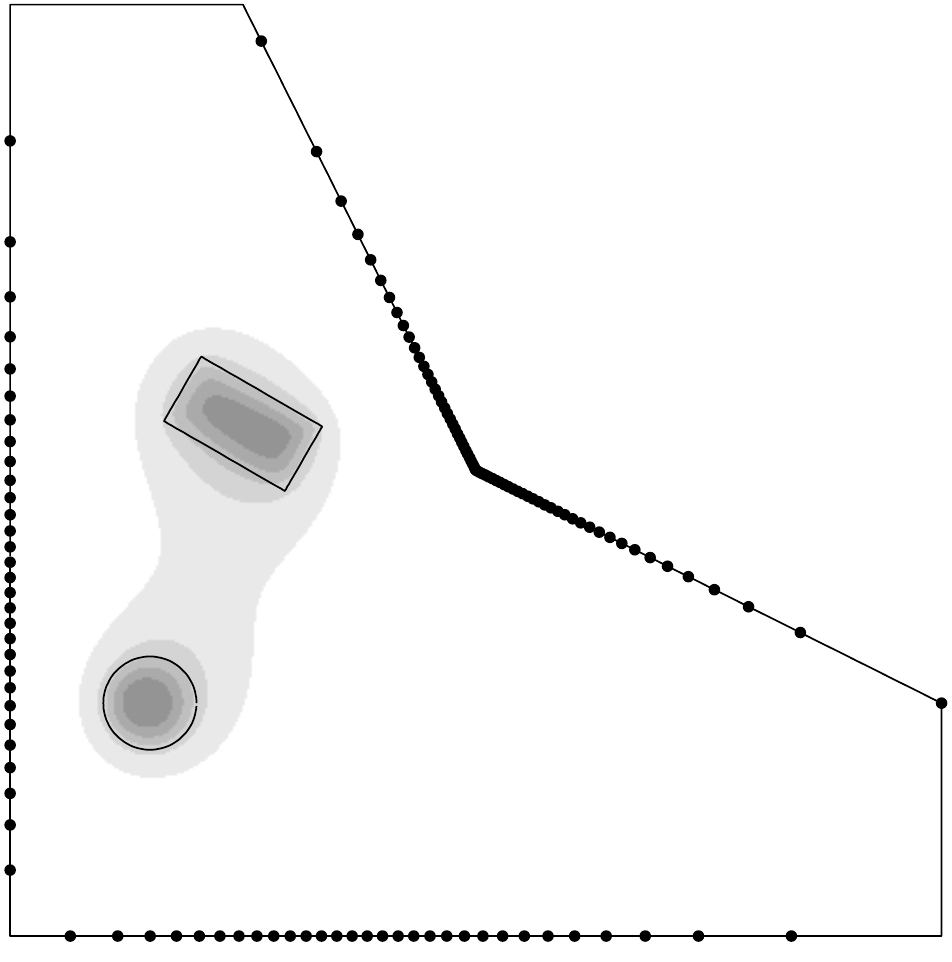}{%
\label{fig:noiseless:reconstruction}
	Reconstruction.
	The actual boundaries of the inclusions are marked with solid lines.
	The different shades of gray correspond to different $C_\infty$.
	Electrodes $x_1,\ldots,x_{128}$ are marked with dots.
}%
\caption{
\label{fig:noiseless}
	Simulation of and reconstruction from noiseless discrete bisweep data.
}
\end{figure}

\begin{figure}
\centering
\incsubfigFF{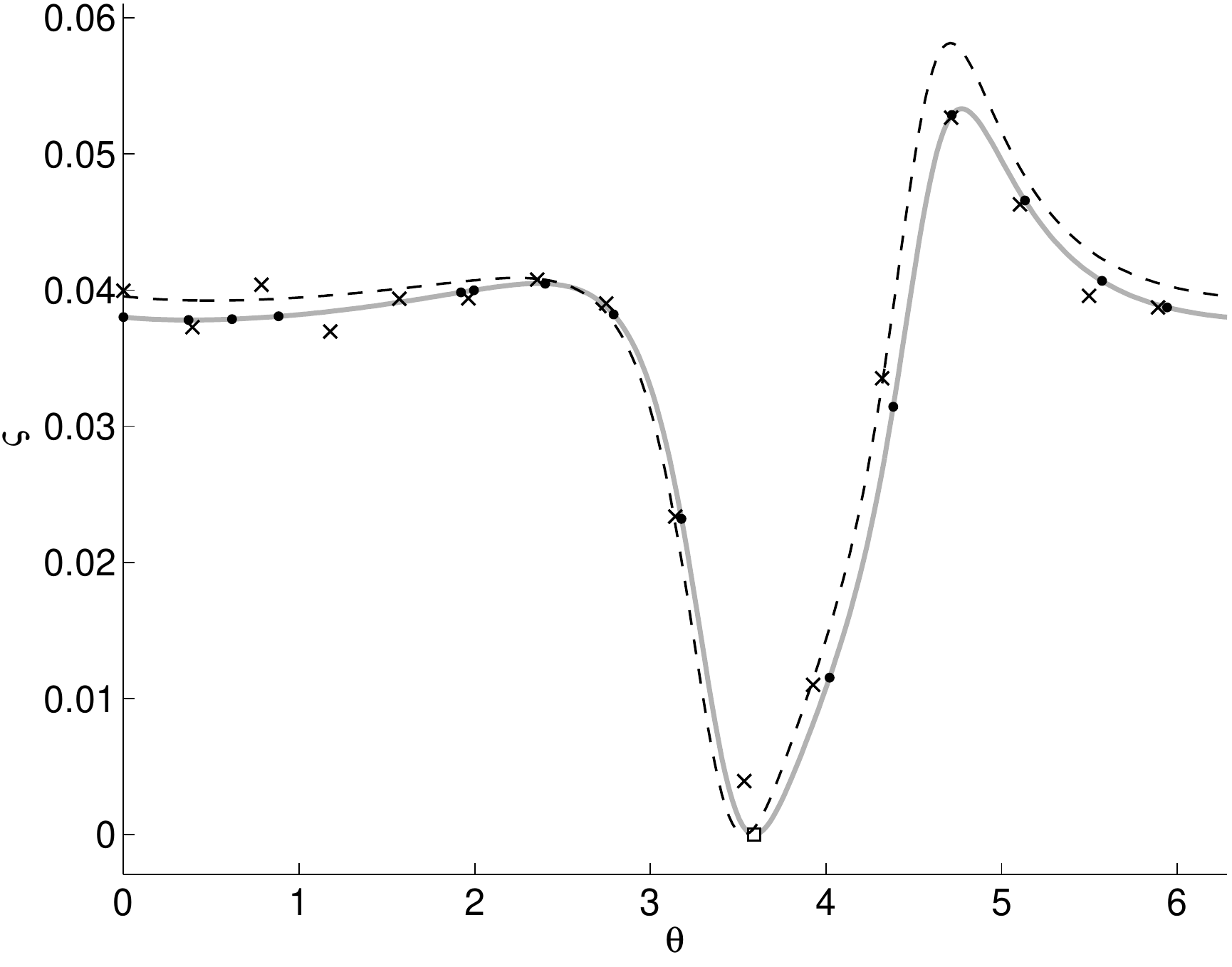}{%
	Sweep data $\varsigma_{\sigma'}(x_1',\cdot)$ (solid line, dots sampled at $x_2',\ldots,x_{n}'$, square at $x_1'$), sweep data $\varsigma_{\sigma}(x_1,\cdot)$ (dashed), and $(x_j,\varsigma_{1j}')$ (x). The noisy data equals that of Fig. \protect\subref{fig:noisy:1}.
	\label{fig:noisy:sweep}
}%
\incsubfigFF{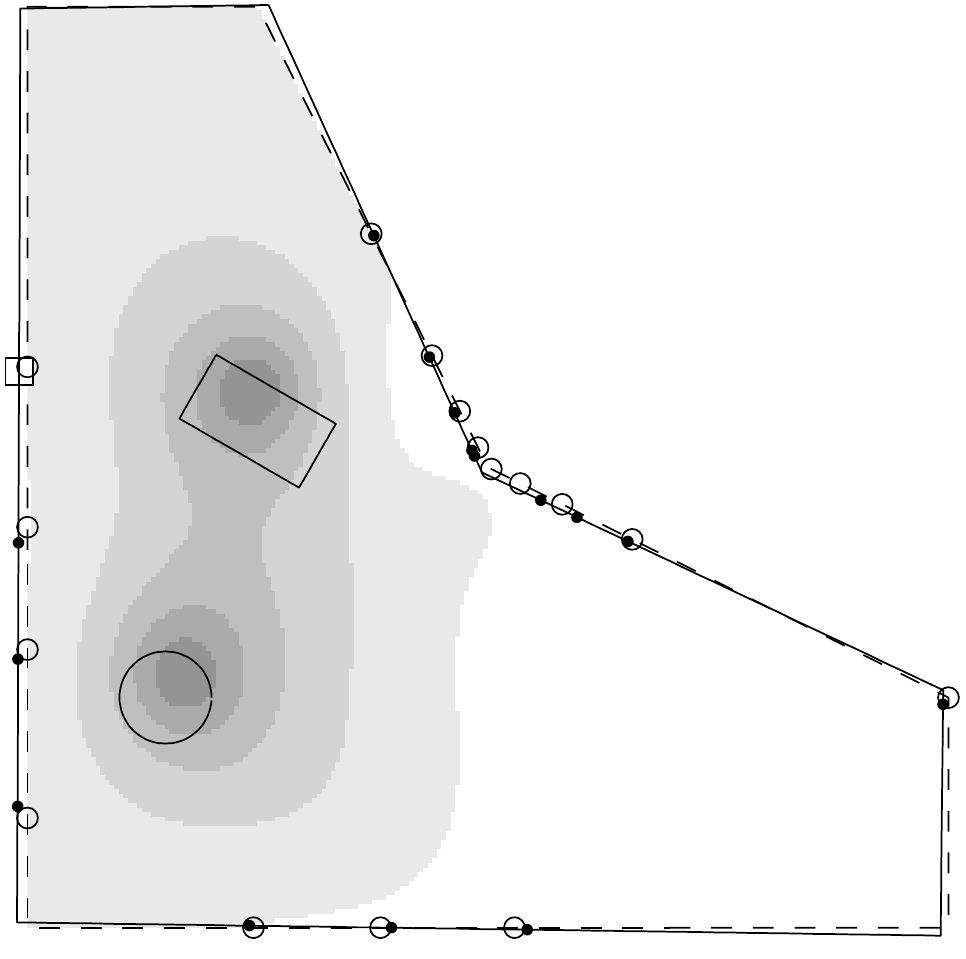}{%
	Geometry: $\partial D'$ and inclusion boundaries (solid), $\partial D$ (dashed);
	electrodes: $x_1,\ldots,x_{n}$ (circles), $x_1'$ (square), and $x_2',\ldots,x_{n}'$ (dots);
	and reconstruction (gray levels).
	\label{fig:noisy:1}
}%
\\[\subfigvsep]%
\incsubfigFF{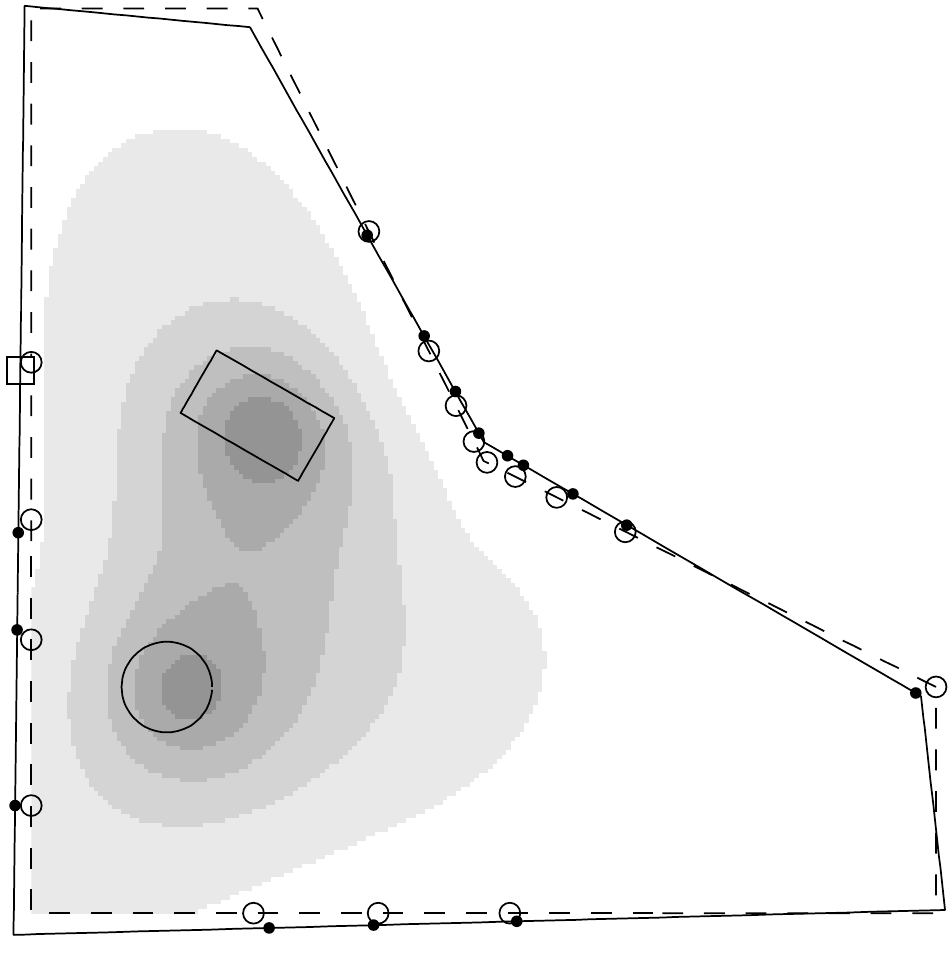}{}%
\incsubfigFF{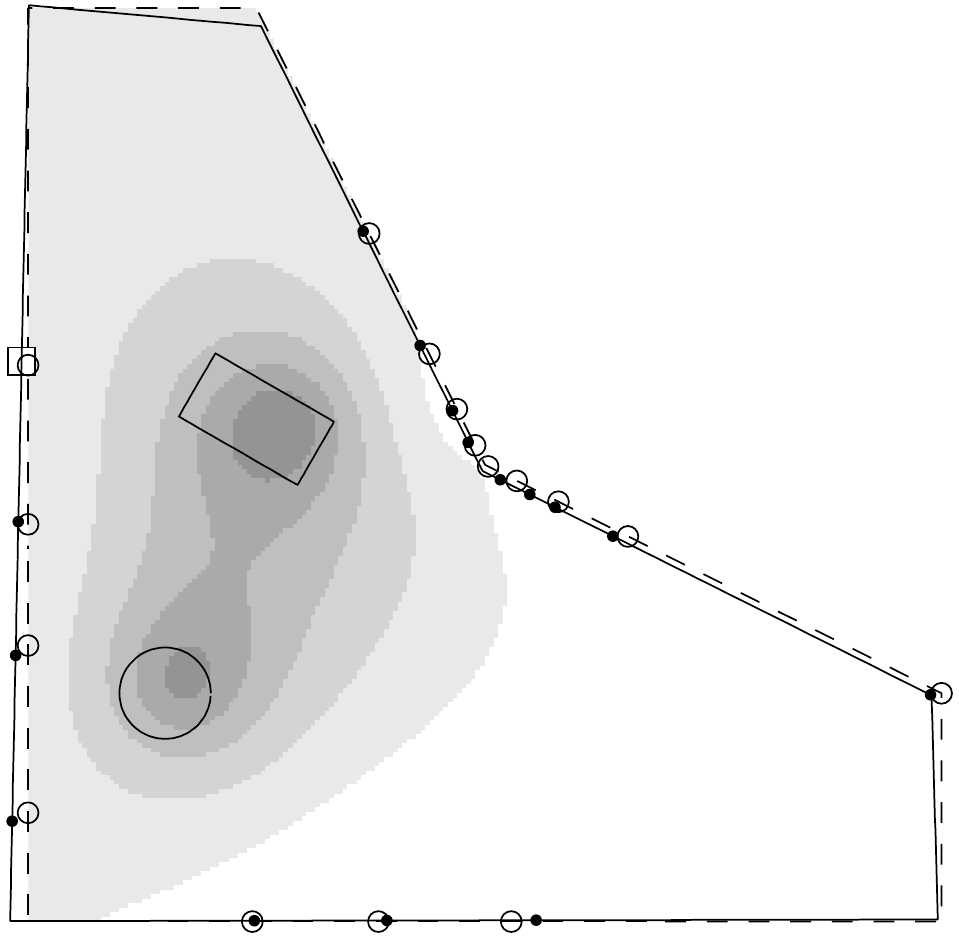}{}%
\\[-1em]%
\incsubfigFF{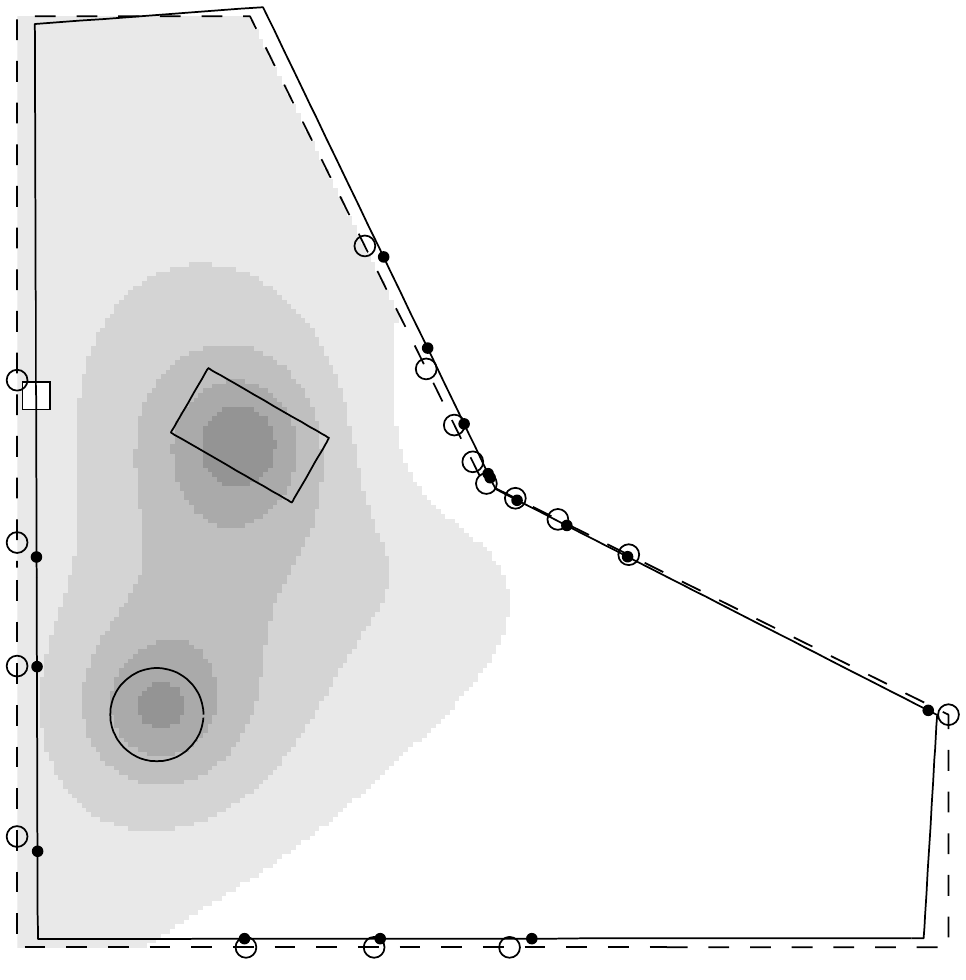}{}%
\incsubfigFF{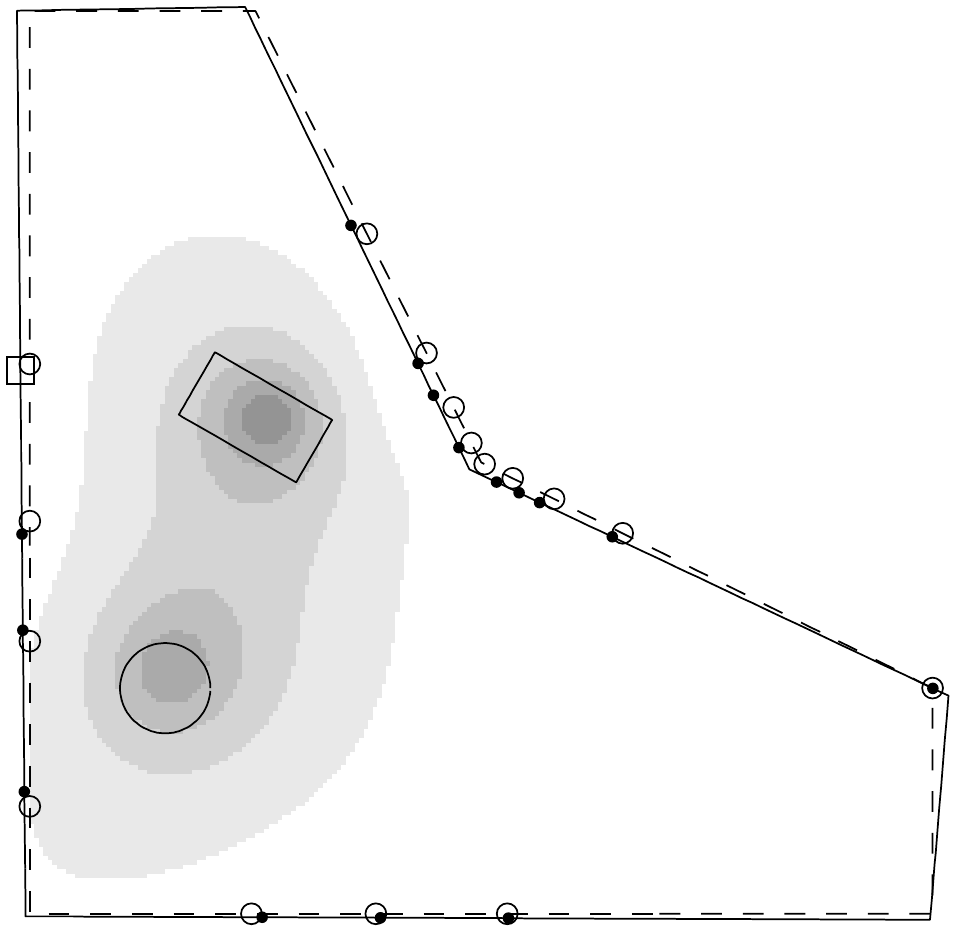}{\label{fig:noisy:last}}%
\caption{
\label{fig:noisy}
	\protect\subref{fig:noisy:1}--\protect\subref{fig:noisy:last}: Reconstructions from noisy data. 
	The different shades of gray correspond to different $C_\infty$ (log. spacing).
}%
\end{figure}

\clearpage

\appendix
\section{Conformal maps in piecewise smooth domains}\label{appendix:conformal:maps}

Theorem \ref{thm:conformal:sharp:angle:intermediary} shows how any piecewise $\cont^{1,\alpha}$ smooth plane domain can be conformally mapped to the unit disk $B$ through an intermediate (piecewise smooth) domain $D'$ so that the maps $D' \to D$ and $D' \to B$ are $\cont^{1,\beta}$ smooth. This is based on how piecewise smooth domains transform under fundamental conformal maps $z \mapsto z^\alpha$, as stated by Lemma~\ref{lemma:piecewise:smoothness:preserved}. Lemma~\ref{lemma:conformal:substitution} justifies the use of conformal transplantation with weak Neumann problems.

\begin{lemma}\label{lemma:piecewise:smoothness:preserved}
Let $\alpha \in (0,2)$ and $D$ be a piecewise $\cont^{1,\beta}$ smooth plane domain such that $0 \in \partial D$ and the mapping $z \mapsto z^\alpha$ (with some branch cut) is continuous and injective in $\overline D$.
The mapped domain $D^\alpha$ is piecewise $\cont^{1,\beta'}$ smooth for $\beta' = \min(\beta,\beta/\alpha) > 0$. Moreover, if $D$ has an internal angle of $\gamma$ at $0$ and $\alpha = \frac{\pi}\gamma$, then $D^\alpha$ is $\cont^{1,\beta'}$ smooth near the origin.  
\end{lemma}

The above can be proven easily using a natural parametrization of the boundary $\partial D$.

\begin{theorem}\label{thm:conformal:sharp:angle:intermediary}
Let $D$ be a piecewise $\cont^{1,\alpha}$ smooth plane domain and $\Phi : D \to B$ a conformal map.
There exists a piecewise $\cont^{1,\alpha}$ smooth plane domain $D'$ and conformal maps $\Phi_1 : D' \to B$, $\Phi_2 : D' \to D$ such that $\Phi_1 \in \cont^{1,\alpha'}(\overline D')$ for some $\alpha' > 0$, $\Phi_2 \in \cont^\infty(\overline D')$ and $\Phi = \Phi_1 \circ \Phi_2^{-1}$.
\end{theorem}

\begin{proof}
The domain $D'$ can be constructed as follows: pick a corner $x \in \partial D$ with a reflex internal angle $\gamma > \pi$ (if any) and some point $x'$ outside $\overline D$ such that the line segment between $x$ and $x'$ does not intersect $D$. The conformal map $z \mapsto 1/(z-x') \in \cont^\infty(\overline D)$ yields a piecewise $\cont^{1,\alpha}$ smooth domain $D_i$ with an internal angle $\gamma$ at $\tilde x = 1/(x-x')$, and there exists an open ray from $\tilde x$ to infinity that does not intersect $\overline D_i$. Using this ray as a branch cut, one may apply the transformation $z \mapsto (z-\tilde x)^{1/2}$, which yields a domain $D_v$ that is piecewise $\cont^{1,\alpha}$ smooth and has one reflex angle less than $D$. This process may be repeated until no internal angle is reflex. The resulting domain is $D'$, and the chained (inverse) transformation $\Phi_2$ is $\cont^\infty(\overline{D'})$ by explicit construction.

A conformal map $\tilde \Phi : D' \rightarrow B$ can be constructed as a composition $\Phi_3 \circ \Phi_4 =: \Phi_3 \circ \tilde \Phi_m \circ \cdots \circ \tilde \Phi_1$, where $\tilde \Phi_j : D'_{j-1} \to D'_j$,
\[
	\Phi_j(z) = (\phi_j(z))^{\pi/\gamma_j}, \qquad
	\phi_j \in \cont^\infty(\overline{ D_{j-1}'}), \qquad
	j = 1,\ldots,m,
\]
straighten the angles $\gamma_j < \pi$ of $D' =: D_0'$. Thus, for some $\alpha' > 0$, $\Phi_4 \in \cont^{1,\alpha'}(\overline{D'})$, and the map $\Phi_3 \in \cont^{1,\alpha'}(\overline{D_m'})$ transforms the resulting $\cont^{1,\alpha'}$ smooth domain to the unit disk $B$ (cf. \cite[§§ 3.3 \& 3.4]{pommerenke1992}). 

Finally, any conformal map $\Psi : D \to B$ satisfies $\Phi = M \circ \Psi$, where $M \in \cont^\infty(\overline B)$ is a Möbius transformation, which shows that $\Phi = M \circ \tilde \Phi \circ \Phi_2^{-1} =: \Phi_1 \circ \Phi_2^{-1}$ is as claimed.
\end{proof}

\begin{lemma}\label{lemma:conformal:substitution}
Let $\tilde f, \tilde g \in H^1_{\rm loc}(\tilde D)$ be such that $\nabla \tilde f \cdot \nabla \tilde g \in L^1(\tilde D)$. Define $f = \tilde f \circ \Phi$, $g = \tilde g \circ \Phi$ and $\sigma = \tilde \sigma \circ \Phi$, where $\Phi : D \to \tilde D$ is conformal, $D$, $\tilde D$ Lipschitz, and $\tilde \sigma \in L^\infty(\tilde D)$. Then $\nabla f \cdot \nabla g \in L^1(D)$ and
\begin{equation*}
	\int_D \sigma \nabla f \cdot \nabla g \d x = \int_{\tilde D} \tilde \sigma \nabla \tilde f \cdot \nabla \tilde g \d x.
\end{equation*}
\end{lemma}

\begin{proof}
Let $D_\epsilon = \{ x \in D \;:\; {\rm dist}(x,\partial D) > \epsilon \}$.
Clearly $\Phi|_{D_\epsilon} \in \cont^\infty(\overline D_\epsilon)$ for all $\epsilon > 0$ and thus %
\begin{equation}
\label{eq:conformal:substitution:proof}
\begin{split}
	\int_{D_\epsilon} |\sigma \nabla f \cdot \nabla g| \d x
	&= \int_{D_\epsilon} |(\tilde \sigma \circ \Phi) \nabla (\tilde f \circ \Phi) \cdot \nabla (\tilde g \circ \Phi)|  \d x \\
	&= \int_{D_\epsilon} |(\tilde \sigma \circ \Phi) \big(J_\Phi J_\Phi^T(\nabla \tilde f) \circ \Phi\big) \cdot \big((\nabla \tilde g) \circ \Phi\big)| \d x \\
	&= \int_{\Phi(D_\epsilon)} \left|(\det J_\Phi \circ \Phi^{-1})\tilde \sigma \nabla \tilde f \cdot \nabla \tilde g\right| \cdot |\det J_{\Phi^{-1}}| \d x \\
	&= \int_{\Phi(D_\epsilon)} |\tilde \sigma \nabla \tilde f \cdot \nabla \tilde g| \d x, 
\end{split}
\end{equation}
where $J_\Phi : D \rightarrow \R^{2\times 2}$ denotes the Jacobian of $\Phi$ (interpreted as a mapping $D \subset \R^2 \rightarrow \R^2$). Since $\Phi$ is conformal, $J_\Phi J_\Phi^T = J_\Phi^T J_\Phi = (\det J_\Phi) I = (\det J_{\Phi^{-1}} \circ \Phi)^{-1}I$.
The monotone convergence theorem yields $\| \sigma \nabla f \cdot \nabla g \|_{L^1(D)} = \| \tilde \sigma \nabla \tilde f \cdot \nabla \tilde g \|_{L^1(\tilde D)}$.
Similarly to \eqref{eq:conformal:substitution:proof}, it holds that $\int_{D_\epsilon} \sigma \nabla f \cdot \nabla g \d x = \int_{\Phi(D_\epsilon)} \tilde \sigma \nabla \tilde f \cdot \nabla \tilde g \d x$ and the claim follows from Lebesgue's dominated convergence theorem.
\end{proof}

\bibliographystyle{acm}
\bibliography{ns}{}

\end{document}